\documentclass[11pt,a4paper,reqno]{amsart}
\usepackage[english]{babel}
\usepackage[T1]{fontenc}
\usepackage{verbatim}
\usepackage{palatino}
\usepackage{amsmath}
\usepackage{mathabx}
\usepackage{amssymb}
\usepackage{amsthm}
\usepackage{amsfonts}
\usepackage{graphicx}
\usepackage{esint}
\usepackage{color}
\usepackage{mathtools}
\usepackage{overpic}

\usepackage[colorlinks = true, citecolor = black]{hyperref}
\author{Amir Algom}
\address{Department of Mathematics, University of Haifa at Oranim, Tivon 36006, Israel}
\email{\href{mailto:amir.algom@math.haifa.ac.il}{amir.algom@math.haifa.ac.il}}

\title{{ergodic  $\times p$}-invariant  measures on $\mathbb{T}^2$ with no dimension dropping projections}
\date{\today}
\thanks{A.A. is supported by the Israel Science Foundation (Grant No.~392/25), NSF--BSF Grant No.~2024692, and Grant No.~2022034 from the United States--Israel Binational Science Foundation (BSF), Jerusalem, Israel.}

\newcommand{\R}{\mathbb{R}}

\newcommand{\N}{\mathbb{N}}

\def\Barint_#1{\mathchoice
          {\mathop{\vrule width 6pt height 3 pt depth -2.5pt
                  \kern -8pt \intop}\nolimits_{#1}}%
          {\mathop{\vrule width 5pt height 3 pt depth -2.6pt
                  \kern -6pt \intop}\nolimits_{#1}}%
          {\mathop{\vrule width 5pt height 3 pt depth -2.6pt
                  \kern -6pt \intop}\nolimits_{#1}}%
          {\mathop{\vrule width 5pt height 3 pt depth -2.6pt
                  \kern -6pt \intop}\nolimits_{#1}}}

\numberwithin{equation}{section}

\theoremstyle{plain}
\newtheorem{thm}{Theorem}
\numberwithin{thm}{section}
\newtheorem*{"thm"}{"Theorem"}

\newtheorem{lemma}[thm]{Lemma}

\newtheorem{cor}[thm]{Corollary}
\newtheorem{proposition}[thm]{Proposition}
\newtheorem{"proposition"}[thm]{"Proposition"}
\newtheorem{"lemma"}[thm]{"Lemma"}
\newtheorem{question}{Question}

\theoremstyle{definition}

\newtheorem{definition}[thm]{Definition}

\theoremstyle{remark}

\newcommand{\nref}[1]{(\hyperref[#1]{#1})}

\DeclareMathSymbol{\intop}  {\mathop}{mathx}{"B3}

\usepackage{enumitem}
\usepackage[nameinlink,capitalise]{cleveref}
\usepackage{microtype}
\usepackage{mathrsfs}

\newcommand{\T}{\mathbb T}

\newcommand{\dimH}{\dim_{\mathrm H}}
\newcommand{\Ent}{\mathrm H}
\newcommand{\I}{\mathrm I}

\begin{document}

\begin{abstract}
Fix an integer $p\geq 2$ and  $0<s<1$. We construct an ergodic
$\times p$-invariant measure $\mu$ on $\mathbb T^2$ of dimension $s$, such that
every line projection preserves dimension, including  when the corresponding projected IFS has exact overlaps. In fact, our measure assigns  mass $O(w^s)$ to every planar tube of width $w$. The result is sharp at the endpoint $s=1$ as shown by Py\"or\"al\"a, Shmerkin, Suomala and Wu (2025). In contrast, we show that  every  quasi-Bernoulli $T_p$-invariant measure with positive dimension admits a dimension-dropping
 projection.
\end{abstract}
\maketitle

\section{Introduction}
\subsection{Background}

A fundamental result in geometric measure theory is the Marstrand--Mattila projection theorem: for $0\leq k\leq d$, let $G_{d,k}$ denote the Grassmannian of $k$-dimensional linear subspaces of $\mathbb{R}^d$, equipped with its natural $\mathrm{SO}(d)$-invariant probability measure. For $V\in G_{d,k}$, let
$\pi_V:\mathbb{R}^d\to V$
denote the orthogonal projection onto $V$. Now, let  $\nu\in\mathcal{P}(\mathbb{R}^d)$ be a Borel probability measure. The Marstrand--Mattila Theorem asserts that for every $0\leq k\leq d$, for almost every $V\in G_{d,k}$,
\begin{equation}\label{Eq Marstrand}
\dim(\pi_V\nu)=\min \{\dim\nu,k\},
\text{ where }
\dim\nu:=\inf \{\dim_H A:\nu(A)>0 \}.
\end{equation}
See \cite[Theorems 9.7 and 9.9]{mattila1999geometry}. Note that 
$\dim(\pi_V\nu)\leq \min \{\dim\nu,k\}$
since  $\pi_V$ is Lipschitz. Thus, we call a projection $\pi_V$ \emph{dimension preserving} if equality holds in \eqref{Eq Marstrand}, and \emph{dimension dropping} otherwise. For dynamically defined measures, Furstenberg (morally) advocated the principle that exceptional projections to \eqref{Eq Marstrand} could be predicted  by the algebraic structure of the underlying dynamical system. 

In this paper we provide non-trivial instances in which this heuristic fails: We construct dynamically defined measures with no dimension-dropping projections (at all), even though  the underlying dynamics singles out many directions in which one might naturally expect dimension drop to occur.

\subsection{Main results} \label{section: main results}

For an integer $p\ge 2$, denote the times $ p$ map of $\mathbb{T}:=\mathbb{R}/\mathbb{Z}$ via
$$T_p(x)=px\mod 1.$$
By a standard abuse of notation, we also denote by $T_p$ the corresponding map on $\mathbb{T}^2$, that is,
$T_p(x,y)= \left( T_p(x),\, T_p(y) \right) :\T^2\to\T^2.$
Which interpretation is meant will be clear from context. Define  the coding map   $\Pi: \mathcal A_p ^\mathbb{N} \to\T^2$ by 
\begin{equation} \label{eq: full coding}
\Pi(\omega)=\sum_{n=1} ^\infty \frac{\omega_n}{p^n}, \text{ where } \mathcal A_p:=  \{0,\ldots,p-1\}^2.
\end{equation}

Let $\mu \in \mathcal{P}(\mathbb{T}^2)$ be an ergodic $T_p$-invariant measure, with positive dimension $\dim \mu =s >0$. We want to understand for which linear maps $L:\R^2\to\R$ we have $\dimH L\mu< \min\{1,\dimH\mu\}.$ We will focus on the subcritical case $0<s<1$, as the case $s=1$ was studied in \cite{Aleksi2025pablo}, as
we discuss below. Since a linear map on $\R^2$ does not in general descend to $\T^2$, we fix the standard fundamental domain $[0,1)^2\subset\R^2$ and regard our measures on $\T^2$ as  measures on $[0,1)^2$ when taking Euclidean projections. Thus, for a linear map $L:\R^2\to\R$, the notation $L\mu$ refers to the push-forward of this representative of $\mu$ under $L$.

To illustrate our discussion, suppose $\mu$ is a fully supported IID measure. That is, 
 $\mu = (\Pi)_* \mathbf{p}^\mathbb{N}$, where $\mathbf{p}$ is a fully supported (strictly positive) probability vector on $\mathcal A_p$. We furthermore assume that
$0<\frac{H(\mathbf p)}{\log p}=s<1,$
so that $\dimH\mu=s$; here $H(\cdot)$ is the usual Shannon entropy; this formula can be found in e.g. \cite[Chapter 2]{bishop2013fractal}. 
In this setting it is easy to find many dimension-dropping projections.
Indeed, let $a\neq b\in\mathcal A_p$, and let
$V\in G_{2,1}$ be the line perpendicular to $a-b$. Then
$\pi_V(a)=\pi_V(b).$
Thus, let $X$ be an $\mathcal A_p$-valued random variable with law
$\mathbf p$, and let
$Y=\pi_V(X).$
Since  $Y$ is obtained from $X$ by collapsing (at least)
$a$ and $b$, 
$H(Y)<H(X)=H(\mathbf p).$
On the other hand,  $\pi_V\mu=(\Pi_e)_*\mathbf p^{\N}$, where $e\in S^1$ spans $V$ and
\begin{equation} \label{eq: proj coding}
\Pi_e(\omega)=\sum_{n=1}^{\infty}\frac{e\cdot\omega_n}{p^n}.
\end{equation}
Therefore  standard considerations  give
$\dim\pi_V\mu
\leq
\frac{H(Y)}{\log p}
<
\frac{H( X)}{\log p}
=s.$ In fact,  one can similarly show that for every $c\in\mathbb Z^2\setminus\{0\}$ with $\gcd(c_1,c_2)=1$, the rational projection $x\mapsto x\cdot c$
 will be dimension dropping for $\mu$. More generally, Theorem~\ref{cor:quasi-bernoulli-drop} below shows that the milder  \emph{quasi-Bernoulli} property  ensures the existence of a dimension-dropping rational projection. Thus, in this setting, the heuristic described above is borne out: the evident algebraic resonances do indeed produce dimension drop.

Our main result shows that in fact, there are  ergodic $T_p$-invariant measures for which there is no dimension dropping projection.  For a line $\ell\subset\mathbb R^2$ and $w>0$, let
$T(\ell,w)
:=
\left\{
x\in\mathbb R^2:
\operatorname{dist}(x,\ell)<\frac{w}{2}
\right\},$
and call $T(\ell,w)$ a tube of width $w$; for a tube $T$ let $w(T)$ denote its width.

\begin{thm}\label{thm:main}
Let $p\geq 2$ and $0<s<1$. There exist an ergodic $T_p$-invariant
probability measure $\mu\in\mathcal P(\mathbb T^2)$ and a constant
$C\geq 1$ such that
$$\dim\mu=s \text{ and }
\mu(T)\leq C w(T)^s \text{ for every Euclidean tube } T\subset\mathbb R^2.$$ 
In particular, for every $V\in G_{2,1}$, 
$\dim\pi_V\mu=s.$
\end{thm}
Here, as above, we regard $\mu$ as a measure on the standard fundamental
domain $[0,1)^2$ when considering Euclidean tubes.
The restriction to $s<1$ in Theorem~\ref{thm:main} is sharp. Indeed, Py\"or\"al\"a,
Shmerkin, Suomala, and Wu \cite{Aleksi2025pablo} proved that an ergodic
$T_p$-invariant measure of dimension $1$ whose every non-principal
projection has dimension $1$ must be one of the product measures
$\nu\times\operatorname{Leb}_{\mathbb T}$ or
$\operatorname{Leb}_{\mathbb T}\times\nu,$
$\text{ where }
\dim\nu=0.$

The tube estimate in Theorem~\ref{thm:main} is  related to the work of Orponen \cite{Orponen2015}: for every
$0<s<1$, he constructed a compact set $K\subset\mathbb R^2$ with
$0<\mathcal H^s(K)<\infty$
such that
$\mathcal H^s(K\cap T)\lesssim C w(T)^s$
for every tube $T$. Here, $\mathcal H^s$ is the $s$-dimensional Hausdorff measure. So, 
$\frac{\mathcal H^s|_K}{\mathcal H^s(K)}$ is an $s$-dimensional probability measure
satisfying the same optimal tube bound as in Theorem~\ref{thm:main}. We note that Orponen’s result settled a problem posed by Carbery \cite{Carbery2009}. For related work on tube occupancy and tube-nullity, see \cite{CarberySoriaVargas2007, Chen2016, ShmerkinSuomala2015,Aleksi2025pablo}.

Theorem \ref{thm:main} is also related to the theory of self-similar IFS:
recall that a self-similar IFS (Iterated Function System) on $\mathbb R^d$ is a finite
family of contracting similarities
\[
\Phi=\{f_i(x)=r_iO_i x+t_i\}_{i\in\Lambda},
\, |\Lambda|<\infty,\,
0<|r_i|<1,\, O_i\in O(d).
\]
Given a probability vector
$(q_i)_{i\in\Lambda}$, the corresponding self-similar measure is the
unique probability measure $\theta$ satisfying
$\theta=\sum_{i\in\Lambda}q_i(f_i)_*\theta.$
More generally, if $\nu$ is any shift-invariant ergodic measure on
$\Lambda^{\mathbb N}$, its image under the corresponding coding map is a
 measure supported on the corresponding self-similar set. See \cite{BaranySimonSolomyak2023} for a general introduction to this topic.

The measures appearing in this paper are naturally of this form.
Indeed,
consider the planar self-similar IFS
$\Phi
=
\left\{
x\mapsto\frac{x+d}{p}:
d\in\mathcal A_p
\right\},$
whose coding map is $\Pi$ from \eqref{eq: full coding}. If $V\in G_{2,1}$ and $e\in S^1$ spans
$V$, then
$\pi_V(x)=(e\cdot x)e.$
After identifying $V$ isometrically with $\mathbb R$ via
$te\mapsto t$, the projection $\pi_V$ becomes the scalar map
$P_e(x):=e\cdot x.$
Accordingly, projecting $\Phi$ onto $V$ gives the one-dimensional
 self-similar IFS
$\Phi_e
=
\left\{
x\mapsto\frac{x+e\cdot d}{p}:
d\in\mathcal A_p
\right\}.$
Moreover, if $\mu=\Pi_*\nu$, then
$P_e\mu=(P_e\circ\Pi)_*\nu$ from \eqref{eq: proj coding}
is precisely the image of the same symbolic measure $\nu$ under the
coding map of $\Phi_e$. Thus $\pi_V\mu$ and $P_e\mu$ differ only by
the above isometric identification and, in particular, have the same
dimension.

Jordan and Rapaport \cite{jordan2019dimension} proved that, if
$\Phi_e$ satisfies exponential separation, then
$\dim P_e\mu
=
\min\left\{
1,\, \dim \mu
\right\}.$
Here exponential separation means that there exist $c>0$ and infinitely
many $n$ such that
$|\varphi_u(0)-\varphi_v(0)|\geq c^n$
for every pair of distinct words $u,v\in\mathcal A_p^n$, where
$\varphi_u,\varphi_v$ denote the corresponding level-$n$ maps of
$\Phi_e$. By an argument of Shmerkin and Solomyak,  recorded in
\cite{hochman2014self}, if $e$ has irrational slope, then $\Phi_e$
satisfies exponential separation. Hence, for the measure $\mu$ constructed
in Theorem~\ref{thm:main}, 
Jordan--Rapaport already implies
$\dim P_e\mu=\dim \mu$
for every irrational direction $e$. We also note that all measures ($\mu$ and its projections) involved in Theorem \ref{thm:main} are exact dimensional; this follows directly from the work of Feng and Hu \cite{feng2009dimension}.

The remaining interesting directions are therefore the rational ones. 
This discussion leads naturally to the following question.

\begin{question}\label{question:prescribed-exceptional}
Let $p\geq2$, let $0<s<1$, and let
$S\subseteq\mathbb P^1(\mathbb Q)$
be arbitrary, where $\mathbb P^1(\mathbb Q)$ denotes the rational
projective line. Does there exist an ergodic $T_p$-invariant probability
measure $\mu$ on $\mathbb T^2$ such that
\[
\dim\mu=s
\text{ and,   }
\dim P_e\mu<s
\quad\Longleftrightarrow\quad
[e]\in S?
\]
\end{question}
Theorem~\ref{thm:main} shows that $S=\varnothing$ is possible. At the
opposite extreme, a fully supported IID measure  will have
$S=\mathbb P^1(\mathbb Q).$ The question asks how
much flexibility remains between these two extremes.

Theorem \ref{thm:main} also naturally suggests the question, of which  properties do force the existence of dimension-dropping  projections. To this end, recall that a $\sigma$-invariant probability measure $\nu$ on a $\mathcal{A}_p ^\mathbb{N}$ is called \emph{quasi-Bernoulli} if there exists $C\geq1$ such
that
\begin{equation}\label{eq:quasiBernoulli}
C^{-1}\nu([U])\nu([V])
\leq
\nu([UV])
\leq
C\nu([U])\nu([V]), \text{ whenever } U,V\in \bigcup_{n} \mathcal{A}_p ^n \text{ have } \nu([UV])>0.
\end{equation}
Here $[U]$ denotes the
cylinder determined by the word $U$.
 IID measures
are quasi-Bernoulli with $C=1$ on their support; more generally, every stationary
Markov measure is quasi-Bernoulli on its support. Likewise,  every Gibbs measure of a
H\"older potential on a topologically mixing subshift of finite type is quasi-Bernoulli.

Our main result in this direction is the following:

\begin{thm}\label{cor:quasi-bernoulli-drop}
Let $\nu \in \mathcal{P}(\mathcal{A}_p ^\mathbb{N})$ be an ergodic quasi-Bernoulli $\sigma$-invariant measure with positive entropy, and let $\mu=\Pi_* \nu \in \mathcal{P}(\mathbb{T}^2)$.  Then
there exists 
$a\in\mathbb Z^2\setminus\{0\}$ such that
$\dim(P_a \mu)<\dim\mu.$
\end{thm}

To prove Theorem~\ref{cor:quasi-bernoulli-drop}, we use the notion of
bilateral determinism, introduced by Ornstein and Weiss
\cite{OrnsteinWeiss1975}. Roughly speaking, a shift-invariant measure on $\mathcal{A}_p ^\mathbb{Z}$ is bilaterally deterministic if every finite block is almost surely determined by the coordinates outside that block. We show
that failure of bilateral determinism (for the natural extension) forces a  dimension dropping rational projection.
This criterion is sufficient but not necessary, as shown by the example
following Theorem~\ref{thm:predictive-drop}. See Section~\ref{sec:positive}
for the precise definitions and proofs. This discussion leads to the following question:
\begin{question}\label{question:symbolic-characterization}
Can one give a sharp  characterization of the
ergodic $T_p$-invariant measures 
$\mu\in\mathcal P(\mathbb T^2)$ with positive entropy, for which every  projection
preserves dimension? 
\end{question}

\subsection{Prior results and general context}
\label{subsec:prior-context}
A guiding principle behind Furstenberg's  conjectures surrounding the \(\times 2,\times 3\) Conjecture \cite{furstenberg1967disjointness}, is that    exceptional geometric phenomena (like dimension dropping projections) of dynamically defined sets and measures, should arise (only) from algebraic resonances. For example, his sumset conjecture predicted that if \(m,n\geq2\) are multiplicatively independent and \(A,B\subset\mathbb T\) are closed \(T_m\)- and \(T_n\)-invariant sets, respectively, then \(\dim_H(A+B)=\min\{1,\dim_H A+\dim_H B\}\). This is  the same as the diagonal projection of \(A\times B\) having the expected dimension. It was proved in the self-similar setting by Peres and Shmerkin \cite{PeresShmerkin2009} and in full generality, including a measure-theoretic version, by Hochman and Shmerkin \cite{hochman2009local}. A parallel theory, more directly related to the present paper, concerns projections of self-similar measures. Hochman and Shmerkin \cite{hochman2009local} proved that, under the strong separation condition and assuming that the semigroup generated by the orthogonal parts acts minimally on \(G_{d,k}\), every \(k\)-dimensional orthogonal projection has the expected dimension \eqref{Eq Marstrand}. Falconer and Jin \cite{FalconerJin2014} later removed the separation assumption, and, more recently, Algom and Shmerkin \cite{AlgomShmerkin2024} obtained a projection-dependent criterion in terms of the orbit of the projection under the rotation group. Further Furstenberg-type projection results for dynamically defined measures under various structural assumptions can be found in \cite{Bruce2019jin,bruce2022furstenberg,Aleksi2024reso,jordan2019dimension,barany2023scaling,BaranyKaenmakiKolossvary2026,AlgomRodriguezHertzWang2026,Wu2025Countable,algom2024plane}.

On the other hand, Farkas \cite{Farkas2016} proved that self-similar measures always admit dimension dropping projections if the group generated by the orthogonal parts of the self-similar IFS is finite (this includes the fully supported IID measures discussed above). At the same time, the algebraic-resonance heuristic has limitations for properties finer than dimension. Nazarov, Peres and Shmerkin \cite{NazarovPeresShmerkin2009} showed that convolutions of mid-\(a\) and mid-\(b\) Cantor--Lebesgue measures have the expected dimension whenever \(0<a,b<1/2\) and \(\log a/\log b\notin\mathbb Q\), but also constructed, in the supercritical regime, uncountably many parameters for which these convolutions are singular. Similarly, Rapaport \cite{Rapaport2017} constructed a planar self-similar measure with strong separation, dense rotations and dimension greater than \(1\) whose projections have full dimension but are singular for a residual set of directions. In the self-affine setting, exceptional projections can exhibit still more surprising behaviour, including failure of exact dimensionality \cite{FengXie2025,MorrisSert2025} and nontrivial algebraic structure in families of dimension-dropping projections \cite{MorrisSert2025}.

Finally, the immediate motivation for the present work came from Host-type
equidistribution theory. Host \cite{Host1995normal} proved that if
$\mu\in\mathcal P(\mathbb T)$ is ergodic and $T_p$-invariant with positive
entropy, and $\gcd(p,q)=1$, then for $\mu$-almost every $x$ the orbit
$\{T_q^n x:n\geq0\}$
equidistributes for Lebesgue measure. This is a pointwise strengthening of
the positive-entropy measure-rigidity theorem of Rudolph
\cite{Rudolph1990}. An important feature of this circle of ideas is that a
pointwise equidistribution theorem can in turn be used to recover
measure rigidity. In particular, Lindenstrauss
\cite{Lindenstrauss2001} developed Host's method further, obtaining both an extension of Host's
equidistribution theorem and a Host-type proof of the
Rudolph--Johnson rigidity theorem \cite{Johnson1992}. 
Hochman and Shmerkin \cite{HochmanShmerkinNormal} later established
Host's theorem in its natural generality, assuming only that $p$ and
$q$ are multiplicatively independent.

In \cite{AlgomHost,Algom2021} we studied higher-dimensional  versions of this circle of ideas. A key
hypothesis in \cite{Algom2021} is that the ergodic $T_p$-invariant measure
$\mu\in\mathcal P(\mathbb T^2)$ under consideration admits a non-principal projection
$\pi_V$ such that
$\dim \pi_V\mu<\dim\mu.$
This  was the  starting point for the present work.

\subsection{Proof sketch}
Recall that
$\mathcal A_p:=\{0,\ldots,p-1\}^2$, and let $\sigma$ be the left shift on $\mathcal A_p^{\mathbb N}$ . We construct an ergodic
$\sigma$-invariant measure $\nu$ on $\mathcal A_p^{\mathbb N}$ such that its metric entropy is
$h_\nu(\sigma)=s\log p,$
and such that its Euclidean model
$\mu:=\Pi_*\nu$ (recall \eqref{eq: full coding}), 
satisfies the
sharp bound
$\mu(T)\lesssim w(T)^s$
for every Euclidean tube $T$. In particular,
$\dim \pi_V\mu=s=\dim\mu
\text{ for every }V\in G_{2,1}.$

Before going into any further details, let us explain the general idea of the proof. It uses  a classical construction of Grillenberger
\cite{Grillenberger1973}, which is typically applied to build minimal uniquely
ergodic symbolic systems with a prescribed  entropy. In its basic
form, one begins with a finite library of words over $\mathcal{A}_p$, and
iteratively replaces it by libraries of much longer words obtained by
carefully chosen concatenations of the previous ones. By controlling the
frequencies with which the old words occur, while retaining exponentially
many admissible concatenations, one obtains a uniquely ergodic subshift  with
the desired entropy. In the present argument, however, we ask substantially
more from the libraries: besides carrying the required entropy, they must
remain uniformly non-concentrated after projection in every direction, and
this uniformity has to persist at all relevant internal windows. It is not
a priori clear that the usual Grillenberger block-construction can be run while
preserving such a large family of geometric constraints.

To achieve this, we use an idea inspired by the lower-bound counterexample in the recent breakthrough work of Orponen and Rutar on the visibility conjecture \cite{OrponenRutar2026}. A useful idea in their construction is that one can trade an arbitrarily small amount of ``extra'' dimension for greater uniformity in geometric parameters.  We implement the same
principle in our setting by working, at the $k$-th stage, with an
exponent $t_k$ which is slightly larger than the final target $s$, with
$t_k\downarrow s$. The surplus $t_k-s$ acts as a budget: when we pass to
 concatenations and  prune the resulting family in order to
enforce the projection estimates simultaneously in every direction and at
every internal window, we allow a small deterioration of the exponent. By
choosing the scales sufficiently far apart, this loss is kept smaller than
$t_k-t_{k+1}$. The construction can therefore be iterated while preserving both the frequency and cardinality estimates needed for unique ergodicity and entropy, and the uniform geometric non-concentration needed for the tube estimate.

Before giving more details, let us fix and recall some  notation. 
For $m\in\mathbb N$, let $\mathcal D_m$ denote $p^m$-adic partition of
$\mathbb R$, 
$\mathcal D_m
:=
\left\{
\left[\frac{k}{p^m},\frac{k+1}{p^m}\right)
:
k\in\mathbb Z
\right\}.$
For $e\in S^1$, recall that
$P_e(x)=e\cdot x.$
If $V=\operatorname{span}\{e\}$, then $P_e$ differs from the orthogonal
projection $\pi_V$ only by a natural isometric identification of $V$
with $\mathbb R$, which preserves dimension. By the standard estimate \eqref{eq:main-occupancy-target} proved below, it suffices to prove that there exists a constant \(C>0\) such that, for every \(m\in\mathbb N\),
\[
\sup_{e\in S^1}\max_{I\in\mathcal D_m}(P_e\mu)(I)
\leq Cp^{-sm}.
\]

\medskip

\paragraph{\textbf{Step 1: the initial library.}}
Fix $0<s<1$. We begin with a finite-scale version of the desired property.
A \emph{library} is a set
$\mathcal W\subseteq\mathcal A_p^N$
of words of length $N$. We will construct such a library
with exponentially many words, while requiring that every internal block of
every possible length remains uniformly well spread out upon projection in every
direction.

More precisely, for $w=(w_1,\ldots,w_N)\in\mathcal W$, a starting position
$0\leq a\leq N-m$, and $1\leq m\leq N$, define
\[
X_{a,m}(w)
:=
\sum_{r=1}^m p^{-r}w_{a+r}
=
\Pi\bigl(w_{a+1},\ldots,w_{a+m}\bigr)
\in[0,1]^2.
\]
Thus $X_{a,m}(w)$ is the Euclidean point encoded  by the first
$m$ digits of the shifted word $\sigma^a w$.

We want these points to be uniformly non-concentrated after projection,
simultaneously for every internal window and every direction. To this end,
fix exponents and constant
\[
s<\alpha<\rho<1, \text{ and } \quad A>0.
\]
We require that for every \(1\leq m\leq N\),
\begin{equation} \label{eq: step 1}
\sup_{\substack{e\in S^1,\;a\in\{0,\ldots,N-m\}\\
I\subset\mathbb R\ \mathrm{interval},\;|I|\leq p^{-m}}}
\frac{1}{|\mathcal W|}
\#\left\{
w\in\mathcal W:
e\cdot X_{a,m}(w)\in I
\right\}
\leq Ap^{-\alpha m}.
\end{equation}

It is not hard to see that such libraries do exist with
$|\mathcal W|\asymp p^{\rho N}$
for all sufficiently large $N$.  The basic observation is that, for a
uniformly chosen word of length $N$, any fixed \(m\)-digit internal block is uniformly
distributed on the \(p^{-m}\)-grid in \([0,1)^2\), while a strip of width
\(p^{-m}\) meets only \(O(p^m)\) of its \(p^{2m}\) points. Thus the
probability of lying in such a strip is \(O(p^{-m})\), substantially smaller
than the permitted bound \(p^{-\alpha m}\). Since \(\rho>\alpha\), a
standard probabilistic argument then produces the required library. See Lemma~\ref{lem:initial-library} for more details.
\medskip

\paragraph{\textbf{Step 2: balanced concatenation and pruning.}}

Suppose that \(\mathcal W\subset\mathcal A_p^N\) satisfies \eqref{eq: step 1} with exponent \(\alpha\), and write \(M:=|\mathcal W|\).
Fix a target exponent $t$ with \(s<t<\alpha\), choose a sufficiently large
multiple \(q\) of \(M\), and set \(N':=qN\). We form length-\(N'\)
concatenations of \(q\) words from \(\mathcal W\), requiring that every
old word occur exactly \(q/M\) times in each concatenation. This
\emph{balancing} will later provide the frequency control needed for
unique ergodicity. We  seek to prune this balanced family to exponential growth rate
\(t\), while retaining its frequency and non-concentration properties.

To see why this is possible, first ignore the balancing condition and
choose the \(q\) old words independently and uniformly from
\(\mathcal W\), producing a random concatenation \(W\). Fix a direction, an interval of length \(p^{-m}\), and an
\(m\)-digit window meeting exactly \(R\) "old" words in pieces of lengths
\(\ell_1,\ldots,\ell_R\). If the projection of the entire window lies in
the given interval, then the projection of its first piece must lie in an
interval of length \(O(p^{-\ell_1})\). After fixing this piece and
rescaling, the same argument as in Step 1 applies to each subsequent piece. Since the
old words are chosen independently, \eqref{eq: step 1} gives
$\mathbb P\bigl(e\cdot X_{a,m}(W)\in I\bigr)
\leq
C^R p^{-\alpha(\ell_1+\cdots+\ell_R)}
=
C^R p^{-\alpha m},$
where \(C\geq1\) depends only on the old non-concentration constant.
Moreover, \(R\leq m/N+2\), and hence
$C^R p^{-\alpha m}
\leq
C^2p^{-\left(\alpha-\frac{\log_p C}{N}\right)m}.$
Thus, by choosing the old block length \(N\) sufficiently large, the
concatenation still satisfies a stronger estimate than the desired one
with exponent \(t\); this is precisely where the surplus
\(\alpha-t\) is used.

Imposing the balancing condition produces only a further arbitrarily
small loss when \(q\) is sufficiently large. The remaining surplus then
allows us, by a standard probabilistic selection, to choose a balanced
sublibrary 
\(\mathcal W'\subset\mathcal A_p^{N'}\), such that for every \(1\leq m\leq N'\)
\begin{equation}\label{eq: step 2}
|\mathcal W'|=\lfloor p^{tN'}\rfloor
\quad\text{and}\quad
\sup_{\substack{e\in S^1,\;0\leq a\leq N'-m\\
I\subset\mathbb R\ \mathrm{interval},\;|I|\leq p^{-m}}}
\frac{1}{|\mathcal W'|}
\#\left\{
W\in\mathcal W':
e\cdot X_{a,m}(W)\in I
\right\}
\leq A'p^{-tm}.
\end{equation}
The formal statement, including the
stronger estimate for windows crossing a single old boundary needed
later, is Theorem~\ref{thm:pruning}.

\medskip

\paragraph{\textbf{Step 3: iteration and the limiting subshift.}}

Choose a decreasing sequence
\[
t_1>t_2>\cdots>s,
\qquad
t_k\downarrow s.
\]
Starting from the library supplied by Step~1, with exponent larger than
\(t_1\), we apply Step~2 successively with target exponents
\(t_1,t_2,\ldots\). At each stage, the new block length is chosen
sufficiently large that the loss incurred at the next concatenation step
is smaller than \(t_k-t_{k+1}\). Thus the surplus above \(s\) provides
enough room to continue the construction indefinitely.
Thus , we obtain lengths \(N_k\to\infty\), constants \(A_k\),
and libraries
\(\mathcal W_k\subset\mathcal A_p^{N_k}\) such that, for every
\(1\leq m\leq N_k\),
\begin{equation}\label{eq: step 3}
|\mathcal W_k|=\lfloor p^{t_kN_k}\rfloor
\quad\text{and}\quad
\sup_{\substack{e\in S^1,\;0\leq a\leq N_k-m\\
I\subset\mathbb R\ \mathrm{interval},\;|I|\leq p^{-m}}}
\frac{1}{|\mathcal W_k|}
\#\left\{
W\in\mathcal W_k:
e\cdot X_{a,m}(W)\in I
\right\}
\leq A_kp^{-t_km}.
\end{equation}
Moreover, every word in \(\mathcal W_{k+1}\) is a balanced concatenation
of words from \(\mathcal W_k\), as in \eqref{eq: step 2}. At each stage we also retain the more refined conclusions of
Theorem~\ref{thm:pruning}: the all-window estimate has a small exponent
surplus beyond \(t_k\), while a window crossing a single boundary between
two "old" words is controlled, up to the pruning error, by the product
of the bounds for its two pieces. Both estimates will be used in Step~4.

For each \(k\), let \(X_k\) be the shift-invariant closed set generated
by infinite concatenations of words from \(\mathcal W_k\), and set
$X:=\bigcap_{k\geq1}X_k.$
The sets \(X_k\) are nested because every level-\((k+1)\) word is a
concatenation of level-\(k\) words. Moreover, balancing ensures that every
level-\((k+1)\) word contains each level-\(k\) word with exactly the same
frequency. Thus, for any fixed finite word, its frequencies in
higher-level words can differ only through occurrences crossing
level-\(k\) boundaries, whose proportion is \(O(1/N_k)\). These
frequencies therefore converge uniformly, so \(X\) is uniquely ergodic;
denote its invariant measure by \(\nu\).

Finally, \eqref{eq: step 3} gives
$h_\nu(\sigma)
\leq
\lim_{k\to\infty}\frac{\log|\mathcal W_k|}{N_k}
=
s\log p.$
The reverse inequality will follow in Step~4 from the projection
non-concentration estimate. The iterative construction of the libraries
and the uniquely ergodic limiting subshift is carried out formally in
Theorem~\ref{thm:fusion}.

\medskip

\paragraph{\textbf{Step 4: intermediate scales and the tube estimate.}}

It remains to pass from the finite libraries to the invariant measure
\(\nu\).  Fix \(m\) and
choose \(k\) such that
$N_{k-1}<m\leq N_k.$
Let \(Q_{k+1,m}\) be the distribution obtained by choosing a word from
\(\mathcal W_{k+1}\) uniformly and then choosing an \(m\)-digit window
inside it uniformly. Since every level-\((k+1)\) word is a concatenation
of level-\(k\) words and \(m\leq N_k\), each such window is of one of two
types: it either lies inside a single level-\(k\) word or crosses exactly
one level-\(k\) boundary.

For windows of the first type, balancing reduces the average over
\(\mathcal W_{k+1}\) to the uniform average over \(\mathcal W_k\), where
the refined all-window estimate applies. For windows of the second type,
we split the window at the boundary and use the product estimate from
Theorem~\ref{thm:pruning}. The exponent surplus above \(s\), together
with the rapid growth of the scales, absorbs the constants and pruning
errors in both cases. Averaging over the possible positions of the
boundary therefore gives
$Q_{k+1,m}
\left(
\left\{
u\in\mathcal A_p^m:
e\cdot X_{0,m}(u)\in I
\right\}
\right)
\lesssim p^{-sm}$
uniformly in \(e\in S^1\) and intervals \(I\) of length at most
\(p^{-m}\).

The balancing condition also ensures that the distribution of
\(m\)-digit blocks under \(\nu\) differs from \(Q_{k+1,m}\) only through
windows crossing level-\((k+1)\) boundaries. Their proportion is
\(O(m/N_{k+1})\), which is negligible by the choice of the scales.
Thus,
\begin{equation}\label{eq: step 4}
\sup_{\substack{e\in S^1\\
I\subset\mathbb R\ \mathrm{interval},\;|I|\leq p^{-m}}}
\nu\left(
\left\{
\omega:
e\cdot\Pi(\omega)\in I
\right\}
\right)
\lesssim p^{-sm}
\end{equation}
for every \(m\), where here we abuse notation
$\Pi(\omega):=\sum_{r=1}^m p^{-r}\omega_r.$
It is now direct to deduce the required estimate  for $\mu:=\Pi_*\nu$
$\sup_{e\in S^1}\max_{I\in\mathcal D_m}(P_e\mu)(I)
\lesssim p^{-sm},$
By Lemma~\ref{lem:padic-to-tubes}, this gives
$\mu(T)\lesssim w(T)^s$
for every Euclidean tube \(T\).

Estimate \eqref{eq: step 4} also gives the missing lower bound
\(h_\nu(\sigma)\geq s\log p\), while Step~3 gives the reverse inequality.
Hence \(h_\nu(\sigma)=s\log p\). The coding then gives
\(\dim\mu\leq s\), whereas the tube estimate gives
\(\dim\mu\geq s\) and \(\dim P_e\mu\geq s\) for every \(e\in S^1\).
Since projections are Lipschitz, we conclude that
\[
\dim\mu=\dim P_e\mu=s
\text{ for every }e\in S^1.
\]
Exact-dimensionality of $\mu$ and $P_e \mu$ follow from
Feng and Hu \cite{feng2009dimension}. The details of the
intermediate-scale estimate and the passage to \(\nu\) are given in
Theorem~\ref{thm:fusion} and Corollary~\ref{cor:fusion-entropy}.

\subsection{Organization}

Section~\ref{sec:prelim} collects the notation and elementary entropy
facts used throughout the paper. In Section~\ref{sec:proof-main} we prove
Theorem~\ref{thm:main}. We only state the initial-library and pruning inputs at this point,
and then prove Theorem \ref{thm:main} using them as a black box. Section~\ref{sec:library-pruning-proofs} then contains the proof of
these statements. Finally, Theorem \ref{cor:quasi-bernoulli-drop} is proved in Section \ref{sec:positive}.

\subsection{Acknowledgements}
An early attempt to solve the problem studied in this paper was carried out jointly with Zhiren Wang at Penn State in 2019--2020. Although we did not resolve it at the time, our work together played an important role in the development of this project. I greatly enjoyed working with Zhiren and learned a great deal from him. I am also grateful to Tuomas Orponen and Meng Wu for helpful comments and discussions.

\section{Preliminaries}\label{sec:prelim}
All logarithms in this paper are natural, unless the base is explicitly indicated.
Also, recall that for $m\in\mathbb N$, we defined
$\mathcal D_m
:=
\left\{
[kp^{-m},(k+1)p^{-m})
:
k\in\mathbb Z
\right\}$.
Next, let $(Y,\mathcal B)$ be a measurable space, let $\theta$ be a probability
measure on $Y$, and let $\mathcal Q$ be a finite or countable measurable
partition of $Y$. We write the Shannon entropy of $\theta$ with respect to $\mathcal{Q}$ as
$\mathrm H(\theta,\mathcal Q)
:=
-\sum_{P\in\mathcal Q}\theta(P)\log\theta(P)\in [0,\infty],$
with the convention $0\log0:=0$.

The following simple observation will be useful:
\begin{lemma}\label{lem:padic-to-tubes}
Let $\mu\in\mathcal P(\mathbb T^2)$ and let $0<s\leq1$. Suppose that there is a constant $C>0$
such that, for every $m\in\mathbb N$,
\[
\sup_{e\in S^1}
\max_{I\in\mathcal D_m}
(P_e\mu)(I)
\leq
C p^{-sm}.
\]
Then there is a constant $C'>0$, depending only on $C,p$, and $s$, such
that
$\mu(T)\leq C' w(T)^s$
for every Euclidean tube $T\subset\mathbb R^2$.
\end{lemma}

\begin{proof}
First, given a bounded interval $J\subset\mathbb R$, if possible choose
$m$ so that
$p^{-(m+1)}<|J|\leq p^{-m}.$
Then $J$ meets at most two atoms of $\mathcal D_m$, and hence, uniformly in
$e\in S^1$,
$(P_e\mu)(J)
\leq
2Cp^{-sm}
\leq
2Cp^s|J|^s.$  If $|J|\geq p^{-1}$, then
$(P_e\mu)(J)\leq1\leq p^s|J|^s.$

Now let $T=T(\ell,w)$ be a tube of width $w$, and let $e\in S^1$ be a
unit normal to $\ell$. Then
$T=P_e^{-1}(J)$
for an interval $J\subset\mathbb R$ of length $w$. Therefore
\[
\mu(T)
=
(P_e\mu)(J)
\leq
2Cp^s w^s.
\]
Thus the conclusion holds with $C'=2Cp^s$.
\end{proof}

We next fix and recall some more notation. Recall that
$\mathcal A_p=\{0,\ldots,p-1\}^2$
and that $\sigma$ denotes the left shift on $\mathcal A_p^{\mathbb N}$.
The coding map $\Pi: \mathcal A_p^{\mathbb N} \rightarrow \mathbb{R}^2$ is defined as in \eqref{eq: full coding}. 
When considering the dynamics on $\mathbb T^2$, we identify
$\Pi(\omega)$ with its image modulo $\mathbb Z^2$, 
so that
$\Pi\circ\sigma=T_p\circ\Pi.$
For $m\geq1$, we also write
$\Pi_m(\omega)
:=
\sum_{n=1}^m p^{-n}\omega_n$
for the $m$-digit truncation. Then,  uniformly in $\omega$,
$\left|\Pi(\omega)-\Pi_m(\omega)\right|
\lesssim_{p}
C p^{-m},$
and hence, uniformly in $e\in S^1$,
$\left|
P_e\Pi(\omega)-P_e\Pi_m(\omega)
\right|
\lesssim_{p} p^{-m}.$

Next, recall that a Borel probability measure $\theta$ on $\mathbb R^d$ is called \emph{exact dimensional} if there exists $s\geq0$ such that \[ \lim_{r\downarrow0} \frac{\log\theta(B(x,r))}{\log r} = s \] for $\theta$-almost every $x\in\mathbb R^d$. In this case, it follows directly that $\dim\theta=s$.

\begin{thm} \cite{feng2009dimension} \label{thm:exact-dimensionality}
Let $\nu$ be an ergodic shift-invariant probability measure on
$\mathcal A_p^{\mathbb N}$, and let
$\mu:=\Pi_*\nu.$
Then $\mu$ is exact dimensional. Moreover, for every $e\in S^1$, the
projected measure $P_e\mu$ is exact dimensional.
\end{thm}
This is a direct consequence of Feng and Hu \cite[Theorem 2.8]{feng2009dimension} applied to the self-similar IFS $\Phi
=
\left\{
x\mapsto\frac{x+d}{p}:
d\in\mathcal A_p
\right\}$ on $\mathbb{R}^2$, and $\Phi_e
=
\left\{
x\mapsto\frac{x+e\cdot d}{p}:
d\in\mathcal A_p
\right\}$ on $\mathbb{R}$.

Finally, if $\theta$ is an exact dimensional compactly supported measure on
$\mathbb R$ of dimension $d$, then it is well known that
\[
\lim_{m\to\infty}
\frac{\mathrm H(\theta,\mathcal D_m)}
     {m\log p}
=
d.
\]
See e.g. \cite{Fan2002measures}.

\section{Proof of Theorem~\ref{thm:main}}\label{sec:proof-main}

Fix $p\geq2$ and $0<s<1$, and recall that
$\mathcal A_p=\{0,\ldots,p-1\}^2.$
We  construct an ergodic $\sigma$-invariant probability measure
$\nu$ on $\mathcal A_p^{\mathbb N}$ and define
$\mu:=\Pi_*\nu.$ We require $\mu$ to satisfy two properties. First, we will arrange that $\dim\mu=s$. Second, we will prove
that there is a constant $C>0$ such that, for every $m\in\mathbb N$,
\begin{equation}\label{eq:main-occupancy-target}
\sup_{e\in S^1}
\max_{I\in\mathcal D_m}
(P_e\mu)(I)
\leq
Cp^{-sm}.
\end{equation}
By Lemma~\ref{lem:padic-to-tubes}, \eqref{eq:main-occupancy-target}
implies
$\mu(T)\lesssim w(T)^s$
for every Euclidean tube $T$. Note that both $\mu$ and all of its projections are 
 exact dimensional by Theorem~\ref{thm:exact-dimensionality}.

\subsection{The initial library, and pruning}\label{subsec:library-pruning}

We begin by recording the two finite-scale inputs used in the construction.
Their proofs are postponed to Section~\ref{sec:library-pruning-proofs}.
They correspond to Steps~1 and~2 of the proof sketch.

\begin{definition} \label{Def windown}
Let $N\in\mathbb N$.

\begin{enumerate}
\item
For a word $w=(w_1,\ldots,w_N)\in\mathcal A_p^N$, we call a pair
$(a,m)$ satisfying
$1\leq m\leq N,
\,
0\leq a\leq N-m,$
an \emph{internal window} of $w$.

\item
For such a window and word, define
$X_{a,m}(w)
:=
\sum_{r=1}^m \frac{w_{a+r}}{p^r}
\in[0,1]^2.$

\item
Fix $C_0=10$. For $\alpha >0, A\geq 1$, we say that a library
$\mathcal W\subset\mathcal A_p^N$ is \emph{$(\alpha,A)$-transverse} if, for every $1\leq m\leq N$,
\begin{equation}\label{eq:transverse-main}
\sup_{\substack{
0\leq a\leq N-m\\
e\in S^1\\
I\subset\mathbb R,\ |I|\leq C_0p^{-m}
}}
\frac{1}{|\mathcal W|}
\#\left\{
w\in\mathcal W:
e\cdot X_{a,m}(w)\in I
\right\}
\leq  A\cdot p^{-\alpha m}.
\end{equation}
\end{enumerate}
\end{definition}

The first input we require says that large transverse libraries exist.

\begin{lemma}\label{lem:initial-library} 
Let
$0<\alpha<\rho<1.$
There exists $A_0\geq1$, depending only on $p,\alpha,\rho$, and $C_0$,
such that for all sufficiently large $N$, there exists an
$(\alpha,A_0)$-transverse library
$\mathcal W\subset\mathcal A_p^N,
\,
|\mathcal W|=\lfloor p^{\rho N}\rfloor.$
\end{lemma}

The second input is the pruning theorem:
\begin{thm}\label{thm:pruning}
There are constants $\Gamma\geq2$ and $C\geq1$, depending only on
$p$ and $C_0$, with the following property.

Suppose that
$\mathcal W\subset\mathcal A_p^N,
\,
M:=|\mathcal W|,$
is $(\alpha,A)$-transverse, and let
\begin{equation}
\label{eq:pruning-budget}
0<t<
\alpha-\frac{\log_p(\Gamma A)}{N}.
\end{equation}

Then, for every sufficiently large multiple $q$ of $M$, setting
$N':=qN, K:=\lfloor p^{tN'}\rfloor,$
there exists a library
$\mathcal W'\subset\mathcal A_p^{N'},
\,
|\mathcal W'|=K,$
with the following properties.

\begin{enumerate}[label=\textup{(\alph*)}]

\item
Every $W\in\mathcal W'$ is a concatenation of $q$ words from
$\mathcal W$, and every $w\in\mathcal W$ occurs exactly $q/M$ times
in every $W\in\mathcal W'$.

\item
There exists $\beta>t$ such that, for every $1\leq m\leq N'$,
\begin{equation}\label{eq:refined-pruning-main}
\sup_{\substack{
0\leq a\leq N'-m\\
e\in S^1\\
I\subset\mathbb R,\ |I|\leq C_0p^{-m}
}}
\frac{1}{K}
\#\left\{
W\in\mathcal W':
e\cdot X_{a,m}(W)\in I
\right\}
\leq
C\left(
A^2p^{-\beta m}
+
\frac{N'}{K}
\right).
\end{equation}

\item
Suppose, in addition, that for some $B\geq1$, $\gamma>0$, and
$\varepsilon\geq0$, for every $1\leq m\leq N$,
\begin{equation}\label{eq:input-refined}
\sup_{\substack{
0\leq a\leq N-m\\
e\in S^1\\
I\subset\mathbb R,\ |I|\leq C_0p^{-m}
}}
\frac{1}{M}
\#\left\{
w\in\mathcal W:
e\cdot X_{a,m}(w)\in I
\right\}
\leq
Bp^{-\gamma m}+\varepsilon.
\end{equation}
Then, for every $1\leq \ell<m\leq N$,
\begin{equation}\label{eq:boundary-pruning}
\sup_{\substack{
1\leq r<q\\
e\in S^1\\
I\subset\mathbb R,\ |I|\leq C_0p^{-m}
}}
\frac{1}{K}
\#\left\{
W\in\mathcal W':
e\cdot X_{rN-\ell,m}(W)\in I
\right\}
\leq
C\left[
\bigl(Bp^{-\gamma\ell}+\varepsilon\bigr)
\bigl(Bp^{-\gamma(m-\ell)}+\varepsilon\bigr)
+
\frac{N'}{K}
\right].
\end{equation}
\end{enumerate}
\end{thm}
In particular, \eqref{eq:refined-pruning-main} implies that
$\mathcal W'$ is $(t,A')$-transverse for some $A'\geq1$ satisfying
\begin{equation}\label{eq: coarse-pruning-main}
A'\leq C(A^2+N').
\end{equation}
Indeed, since
$K=\lfloor p^{tN'}\rfloor\geq \frac12p^{tN'}$
for all sufficiently large $q$, for every $m\leq N'$ we have
$\frac{N'}{K}
\leq
2N'p^{-tN'}
\leq
2N'p^{-tm}.$
Since $\beta>t$, also $p^{-\beta m}\leq p^{-tm}$. Hence
\[
\frac1K
\#\{W\in\mathcal W':e\cdot X_{a,m}(W)\in I\}
\leq
C(A^2+N')p^{-tm}.
\]

We remark that Part~\textup{(b)} gives the inductive estimate
\eqref{eq: coarse-pruning-main}, which is needed for the next stage of
the construction. Part~\textup{(c)} retains additional information for
windows crossing a single old-block boundary; this  estimate
will be used later to control intermediate scales and obtain the tube
bound.

\subsection{Constructing the subshift}
\label{subsec:fusion}

In this Section we construct our
$\sigma$-invariant measure $\nu$. This corresponds to Step 3 in the proof sketch.

Let  $k\in \mathbb{N}$, fix an integer $N_k$,  a library
$\mathcal W_k\subset\mathcal A_p^{N_k}$, and  $1\leq m\leq N_k$. Let
$Q_{k,m}$  be the probability measure on $\mathcal A_p^m$ obtained by
choosing $W\in\mathcal W_k$ uniformly, then choosing
$a\in\{0,\ldots,N_k-m\}$ uniformly, and returning the block
$(W_{a+1},\ldots,W_{a+m}).$

Recall that we fixed  $0<s<1$. Set
\begin{equation} \label{eq:tk}
t_k:=s+(1-s)2^{-k},
\, k\geq1,\text{ and }\delta_k:=t_k-t_{k+1}=(1-s)2^{-k-1}.
\end{equation}
Thus
$t_k\downarrow s.$
\begin{thm}\label{thm:fusion}
There exist integers
$N_1<N_2<\cdots,$
libraries
$\mathcal W_k\subset\mathcal A_p^{N_k},
\,
M_k:=|\mathcal W_k|,$  a constant $C_{\mathrm{fus}}\geq1$, 
and a uniquely ergodic subshift
$X\subset\mathcal A_p^{\mathbb N}$, whose unique invariant probability
measure is denoted by $\nu$, with the following
properties.

\begin{enumerate}[label=\textup{(\roman*)}]

\item\label{item:fusion-cardinality}
For every $k\geq1$,
\begin{equation}\label{eq:fusion-cardinality}
M_k=\lfloor p^{t_kN_k}\rfloor.
\end{equation}
Moreover, there is an integer $q_k$, divisible by
$M_k$, such that
$N_{k+1}=q_kN_k.$
Every $W\in\mathcal W_{k+1}$ is a concatenation of $q_k$ words from
$\mathcal W_k$, and every $w\in\mathcal W_k$ occurs exactly
$q_k/M_k$ times in this concatenation.
\item\label{item:fusion-growth}
For every $k\geq1$,
\begin{equation}\label{eq:fusion-growth}
\frac{N_k}{N_{k+1}}
\leq
p^{-2N_k}.
\end{equation}

\item\label{item:fusion-projections}
For every $k\geq2$ and every
$N_{k-1}<m\leq N_k$,
\begin{equation}\label{eq:fusion-projections}
\sup_{\substack{
e\in S^1\\
I\subset\mathbb R ,\ |I|\leq C_0p^{-m}
}}
Q_{k+1,m}
\left(
\left\{
u\in\mathcal A_p^m:
e\cdot X_{0,m}(u)\in I
\right\}
\right)
\leq
C_{\mathrm{fus}}p^{-sm}.
\end{equation}
\item\label{item:fusion-frequencies}
For every $k\geq1$ and every $1\leq m\leq N_k$,
\begin{equation}\label{eq:fusion-frequencies}
\sup_{\mathcal E\subset\mathcal A_p^m}
\left|
\nu([\mathcal E])-Q_{k,m}(\mathcal E)
\right|
\leq
C_{\mathrm{fus}}\frac{m}{N_k},\, \text{ where } [\mathcal E]
:=
\left\{
\omega\in X:
(\omega_1,\ldots,\omega_m)\in\mathcal E
\right\}.
\end{equation}

\end{enumerate}
\end{thm}

\begin{proof}
We divide the proof into four steps. Recall \eqref{eq:tk}. Let
$C_{\mathrm{pr}}\geq1$ be the constant $C$ in
\eqref{eq: coarse-pruning-main}, and set
$C_*:=2C_{\mathrm{pr}}.$
We construct the libraries so that
\begin{equation}\label{eq:aux-transverse}
\mathcal W_k \text{ is }(t_k,A_k)\text{-transverse},
\, \text{ where }
A_k\leq C_*N_k,
\end{equation}
and, for the constant $\Gamma$ from Theorem~\ref{thm:pruning},
\begin{equation}\label{eq:aux-budget}
\frac{\log_p(\Gamma A_k)}{N_k}
\leq
\frac{\delta_k}{4}.
\end{equation}
We also choose the block lengths $N_k$ sufficiently rapidly increasing so that
\begin{equation}\label{eq:aux-absorption}
C_{\mathrm{pr}}^2C_*^4
N_k^4p^{-(t_{k+1}-s)N_k}
\leq1
\end{equation}
for every $k$.

\medskip

\noindent
\textbf{Step 1: construction of the libraries.}
Choose $\alpha_0,\rho_0$ such that
$t_1<\alpha_0<\rho_0<1.$
By Lemma~\ref{lem:initial-library}, for every sufficiently large $N_0$
there exists an $(\alpha_0,A_0)$-transverse library
$\mathcal W_0\subset\mathcal A_p^{N_0},$ with
$|\mathcal W_0|=\lfloor p^{\rho_0N_0}\rfloor.$
Choose $N_0$ large enough that
\[
\frac{\log_p(\Gamma A_0)}{N_0}
<
\alpha_0-t_1, \text{ so that }
t_1<
\alpha_0-\frac{\log_p(\Gamma A_0)}{N_0}.
\]
Then Theorem~\ref{thm:pruning} applies with target exponent $t_1$.

Choose the concatenation parameter $q_0$, a multiple of
$|\mathcal W_0|$, so large that, with
$N_1:=q_0N_0,$
the resulting library $\mathcal W_1\subset\mathcal A_p^{N_1}$
satisfies
$M_1:=|\mathcal W_1|
=
\lfloor p^{t_1N_1}\rfloor$
and \eqref{eq:aux-budget}--\eqref{eq:aux-absorption} for $k=1$.
By \eqref{eq: coarse-pruning-main}, after increasing $q_0$ if necessary,
\eqref{eq:aux-transverse} also holds for $k=1$.

Suppose now that $\mathcal W_k$ has been constructed. By
\eqref{eq:aux-budget},
\[
\frac{\log_p(\Gamma A_k)}{N_k}
\leq \frac{\delta_k}{4}
<\delta_k,\,
\text{ and therefore }
t_{k+1}
=
t_k-\delta_k
<
t_k-\frac{\log_p(\Gamma A_k)}{N_k}.
\]
Then Theorem~\ref{thm:pruning} applies to $\mathcal W_k$ with target
exponent $t_{k+1}$.

Choose the concatenation parameter $q_k$, a sufficiently large multiple
of $M_k$, and set
$N_{k+1}:=q_kN_k.$
We require
\begin{equation}\label{eq:aux-growth}
\frac{N_k}{N_{k+1}}
\leq
p^{-2N_k},
\end{equation}
as well as
\[
N_{k+1}\geq A_k^2,
\quad
\frac{\log_p(\Gamma C_*N_{k+1})}{N_{k+1}}
\leq
\frac{\delta_{k+1}}{4}, \text{ and }
C_{\mathrm{pr}}^2C_*^4N_{k+1}^4
p^{-(t_{k+2}-s)N_{k+1}}
\leq1.
\]
All these conditions hold once $q_k$ is sufficiently large.

Theorem~\ref{thm:pruning} then gives a library
$\mathcal W_{k+1}\subset\mathcal A_p^{N_{k+1}}$
such that
$M_{k+1}:=|\mathcal W_{k+1}|
=
\lfloor p^{t_{k+1}N_{k+1}}\rfloor.$
Every word in $\mathcal W_{k+1}$ is a concatenation of $q_k$ words from
$\mathcal W_k$, and every word in $\mathcal W_k$ occurs exactly
$q_k/M_k$ times. This proves \ref{item:fusion-cardinality}, while
\eqref{eq:aux-growth} gives \ref{item:fusion-growth}.

By \eqref{eq: coarse-pruning-main},
$\mathcal W_{k+1}$ is $(t_{k+1},A_{k+1})$-transverse for some
$A_{k+1}
\leq
C_{\mathrm{pr}}(A_k^2+N_{k+1}).$
Since $N_{k+1}\geq A_k^2$ and $C_*=2C_{\mathrm{pr}}$, it follows that
$A_{k+1}\leq C_*N_{k+1}.$
Hence, by the requirements imposed on $N_{k+1}$ in \eqref{eq:aux-growth} and the equation following it,
\eqref{eq:aux-transverse}--\eqref{eq:aux-absorption} hold with
$k$ replaced by $k+1$. This completes the inductive construction of the
libraries $\mathcal W_k$.

We shall need two refined consequences of the pruning construction.
First, applying Theorem~\ref{thm:pruning}\textup{(b)} at stage $k$,
for every $k\geq1$ there exists
$\beta_k>t_{k+1}$
such that, setting
\[
B_k:=C_{\mathrm{pr}}A_k^2,
\qquad
\varepsilon_{k+1}
:=
C_{\mathrm{pr}}\frac{N_{k+1}}{M_{k+1}},
\]
we have, for every $1\leq m\leq N_{k+1}$,
\begin{equation}\label{eq:aux-refined}
\sup_{\substack{
0\leq a\leq N_{k+1}-m\\
e\in S^1\\
I\subset\mathbb R,\ |I|\leq C_0p^{-m}
}}
\frac{1}{M_{k+1}}
\#\left\{
W\in\mathcal W_{k+1}:
e\cdot X_{a,m}(W)\in I
\right\}
\leq
B_kp^{-\beta_km}+\varepsilon_{k+1}.
\end{equation}
Moreover, since
$M_{k+1}=\lfloor p^{t_{k+1}N_{k+1}}\rfloor$,
\begin{equation}\label{eq:aux-epsilon}
\varepsilon_{k+1}
\leq
2C_{\mathrm{pr}}N_{k+1}
p^{-t_{k+1}N_{k+1}}.
\end{equation}

Second, for $k\geq2$, apply
Theorem~\ref{thm:pruning}\textup{(c)} at stage $k$, using
\eqref{eq:aux-refined} from stage $k-1$ as input. Then, for every
$1\leq\ell<m\leq N_k$,
\begin{equation}\label{eq:aux-boundary}
\begin{split}
\sup_{\substack{
1\leq r<q_k\\
e\in S^1\\
I\subset\mathbb R,\ |I|\leq C_0p^{-m}
}}
\frac{1}{M_{k+1}}
\#\Bigl\{
W\in\mathcal W_{k+1}:
e\cdot X_{rN_k-\ell,m}(W)\in I
\Bigr\}
\leq C_{\mathrm{pr}}\Bigl[
&
\bigl(B_{k-1}p^{-\beta_{k-1}\ell}+\varepsilon_k\bigr)
\\
&\cdot
\bigl(B_{k-1}p^{-\beta_{k-1}(m-\ell)}
+\varepsilon_k\bigr)
+
\frac{N_{k+1}}{M_{k+1}}
\Bigr].
\end{split}
\end{equation}

\medskip
\noindent
\medskip
\noindent
\textbf{Step 2: comparison with uniform internal windows.}

Fix $k\geq1$, $1\leq m\leq N_k$, and
$\mathcal E\subset\mathcal A_p^m$. Let $V\in\mathcal W_{k+1}$ and
write
$V=W_1\cdots W_{q_k},
\,
W_j\in\mathcal W_k.$
By item~\ref{item:fusion-cardinality}, every
$W\in\mathcal W_k$ occurs exactly $q_k/M_k$ times among
$W_1,\ldots,W_{q_k}$. Hence, by the definition of $Q_{k,m}$,
\begin{align*}
&\#\Bigl\{
0\leq a\leq N_{k+1}-m:
(V_{a+1},\ldots,V_{a+m})\in\mathcal E,\\
&\hspace{40mm}
(V_{a+1},\ldots,V_{a+m})
\text{ is contained in a single }W_j
\Bigr\}
\\
&\qquad=
\frac{q_k}{M_k}
\sum_{W\in\mathcal W_k}
\#\Bigl\{
0\leq a\leq N_k-m:
(W_{a+1},\ldots,W_{a+m})\in\mathcal E
\Bigr\}
\\
&\qquad=
q_k(N_k-m+1)Q_{k,m}(\mathcal E).
\end{align*}

The remaining $m$-windows cross a boundary between two consecutive
level-$k$ blocks. Since $m\leq N_k$, each such window crosses exactly
one boundary. For each of the $q_k-1$ boundaries there are $m-1$
possible crossing windows, so their total number is
$(q_k-1)(m-1).$
Thus
\[
N_{k+1}-m+1
=
q_k(N_k-m+1)+(q_k-1)(m-1),
\]
where the first term counts the internal windows and the second the
boundary-crossing ones. By the preceding computation, the internal $m$-windows in $V$ have
empirical distribution exactly $Q_{k,m}$. The boundary-crossing windows may have arbitrary
distribution, but they constitute a proportion at most
\[
\frac{(q_k-1)(m-1)}{N_{k+1}-m+1}
\leq
\frac{m}{N_k}
\]
of all $m$-windows. Consequently, uniformly in
$V\in\mathcal W_{k+1}$,
\begin{equation}\label{eq:block-distribution}
\sup_{\mathcal E\subset\mathcal A_p^m}
\left|
\frac{1}{N_{k+1}-m+1}
\#\left\{
0\leq a\leq N_{k+1}-m:
(V_{a+1},\ldots,V_{a+m})\in\mathcal E
\right\}
-
Q_{k,m}(\mathcal E)
\right|
\leq
\frac{m}{N_k}.
\end{equation}

\medskip

\noindent
\textbf{Step 3: the limiting subshift and unique ergodicity.}
For each $k$, let $\mathcal C_k$ be the set of all one-sided infinite
concatenations of words from $\mathcal W_k$, and set
$X_k:=\bigcup_{a=0}^{N_k-1}\sigma^a\mathcal C_k.$ 
Each $X_k$ is nonempty and compact. Since every word in
$\mathcal W_k$ has length $N_k$,
$\sigma^{N_k}\mathcal C_k=\mathcal C_k,$
and hence $X_k$ is $\sigma$-invariant.
Moreover, $X_{k+1}\subset X_k$. Indeed, if
$0\leq a<N_{k+1}$, write
$a=jN_k+r,
\,
0\leq r<N_k.$
Since every word in $\mathcal W_{k+1}$ is a concatenation of words
from $\mathcal W_k$,
$\sigma^a\mathcal C_{k+1}
\subset
\sigma^r\mathcal C_k
\subset X_k.$
Thus the $X_k$ form a nested sequence of nonempty compact
$\sigma$-invariant sets. Hence
\[
X:=\bigcap_{k\geq1}X_k
\]
is a nonempty compact $\sigma$-invariant subshift.

We now prove that $X$ is uniquely ergodic. Let $\eta$ be an ergodic
$\sigma$-invariant probability measure on $X$, and let
$\omega\in X$ be generic for $\eta$. Fix $k\geq1$,
$1\leq m\leq N_k$, and
$\mathcal E\subset\mathcal A_p^m$. Since $\omega\in X_{k+1}$,
after deleting an initial segment of length less than $N_{k+1}$,
the remaining sequence is a concatenation of words from
$\mathcal W_{k+1}$.

Apply \eqref{eq:block-distribution} to each of these level-$(k+1)$
words. The $m$-windows crossing a boundary between two successive
level-$(k+1)$ words have asymptotic proportion at most
$m/N_{k+1}$. Letting the length of the empirical average tend to
infinity therefore gives
\[
\left|
\eta([\mathcal E])-Q_{k,m}(\mathcal E)
\right|
\leq
\frac{m}{N_k}+\frac{m}{N_{k+1}}
\leq
\frac{2m}{N_k}.
\]

Now fix $m$ and $\mathcal E\subset\mathcal A_p^m$ and let
$k\to\infty$. The right-hand side tends to zero, so
$\eta([\mathcal E])$ is independent of the choice of the ergodic
invariant measure $\eta$. Hence all ergodic $\sigma$-invariant
probability measures on $X$ coincide. So, 
$X$ is uniquely ergodic; denote its unique invariant probability
measure by $\nu$.

\begin{equation}\label{eq:fusion-frequency-bound}
\left|
\nu([\mathcal E])-Q_{k,m}(\mathcal E)
\right|
\leq
2\frac{m}{N_k}.
\end{equation}

\medskip

\noindent
\textbf{Step 4: projected non-concentration at intermediate scales.}
It remains to prove \ref{item:fusion-projections}. Fix $k\geq2$ and
$N_{k-1}<m\leq N_k$. Let $e\in S^1$, let $I\subset\mathbb R$ satisfy
$|I|\leq C_0p^{-m}$, and set
\[
\mathcal E
:=
\left\{
u\in\mathcal A_p^m:
e\cdot X_{0,m}(u)\in I
\right\}.
\]
We estimate $Q_{k+1,m}(\mathcal E)$.

Write each $V\in\mathcal W_{k+1}$ as
$V=W_1\cdots W_{q_k},
\, W_j\in\mathcal W_k.$
In the definition of $Q_{k+1,m}$, we average over
$V\in\mathcal W_{k+1}$ and starting positions
$0\leq a\leq N_{k+1}-m$. Since $m\leq N_k$, the corresponding
$m$-block
$(V_{a+1},\ldots,V_{a+m})$
either lies entirely inside one of the level-$k$ blocks $W_j$, or
crosses a single boundary between two consecutive level-$k$ blocks.
We estimate these two contributions separately.

\smallskip
\noindent
\emph{Internal windows.}
Fix $0\leq a\leq N_k-m$. For $0\leq r<q_k$, the internal $m$-window
starting at $rN_k+a$ satisfies
\[
X_{rN_k+a,m}(V)=X_{a,m}(W_{r+1}).
\]
By the balanced-concatenation property from Step~1, averaging over
$V\in\mathcal W_{k+1}$ and $0\leq r<q_k$ makes the block $W_{r+1}$
uniformly distributed over $\mathcal W_k$. Hence
\[
\begin{split}
&\frac{1}{M_{k+1}q_k}
\sum_{V\in\mathcal W_{k+1}}
\#\left\{
0\leq r<q_k:
e\cdot X_{rN_k+a,m}(V)\in I
\right\}
\\
&\qquad=
\frac{1}{M_k}
\#\left\{
W\in\mathcal W_k:
e\cdot X_{a,m}(W)\in I
\right\}
\\
&\qquad\leq
B_{k-1}p^{-\beta_{k-1}m}+\varepsilon_k,
\end{split}
\]
where in the last line we used \eqref{eq:aux-refined} at stage $k-1$.

Summing over the $N_k-m+1$ possible internal offsets $a$ and
normalizing by the total number $N_{k+1}-m+1$ of $m$-windows, the
 contribution to $Q_{k+1,m}(\mathcal E)$ from internal windows is therefore at most
\[
\frac{q_k(N_k-m+1)}{N_{k+1}-m+1}
\left(
B_{k-1}p^{-\beta_{k-1}m}+\varepsilon_k
\right)
\leq
B_{k-1}p^{-\beta_{k-1}m}+\varepsilon_k.
\]

We now bound the two terms. By the definition of $B_{k-1}$ and
\eqref{eq:aux-transverse},
\[
B_{k-1}
=
C_{\mathrm{pr}}A_{k-1}^2
\leq
C_{\mathrm{pr}}C_*^2N_{k-1}^2,
\qquad
\beta_{k-1}>t_k.
\]
Since $m>N_{k-1}$,
\[
\begin{split}
B_{k-1}p^{-\beta_{k-1}m}
&\leq
C_{\mathrm{pr}}C_*^2N_{k-1}^2
p^{-(t_k-s)N_{k-1}}p^{-sm}
\\
&\leq
p^{-sm},
\end{split}
\]
where the last inequality follows from \eqref{eq:aux-absorption} with
$k-1$ in place of $k$.
Similarly, by \eqref{eq:aux-epsilon},
\[
\begin{split}
\varepsilon_k
&\leq
2C_{\mathrm{pr}}N_kp^{-t_kN_k}
\\
&=
2C_{\mathrm{pr}}N_k
p^{-(t_k-s)N_k}p^{-sN_k}
\\
&\leq
p^{-sN_k}
\leq
p^{-sm},
\end{split}
\]
where we used $t_k>t_{k+1}$, \eqref{eq:aux-absorption}, and
$m\leq N_k$. We conclude that
\begin{equation}\label{eq:internal-bound}
\text{\rm contribution to }Q_{k+1,m}(\mathcal E)
\text{ from internal windows}
\leq
2p^{-sm}.
\end{equation}

\smallskip
\noindent
\emph{Boundary windows.}
For $1\leq r<q_k$ and $1\leq\ell<m$, consider the $m$-window
consisting of the last $\ell$ digits of $W_r$ followed by the first
$m-\ell$ digits of $W_{r+1}$. Its starting position in $V$ is
$rN_k-\ell$. Hence \eqref{eq:aux-boundary} gives
\[
\begin{split}
&\frac{1}{M_{k+1}}
\#\left\{
V\in\mathcal W_{k+1}:
e\cdot X_{rN_k-\ell,m}(V)\in I
\right\}
\\
&\qquad\leq
C_{\mathrm{pr}}\left[
\bigl(B_{k-1}p^{-\beta_{k-1}\ell}+\varepsilon_k\bigr)
\bigl(B_{k-1}p^{-\beta_{k-1}(m-\ell)}+\varepsilon_k\bigr)
+
\frac{N_{k+1}}{M_{k+1}}
\right].
\end{split}
\]

Summing over $1\leq r<q_k$ and $1\leq\ell<m$, and normalizing by the
total number $N_{k+1}-m+1$ of starting positions, gives the boundary
contribution to $Q_{k+1,m}(\mathcal E)$. Since $m\leq N_k$,
\[
\frac{q_k-1}{N_{k+1}-m+1}\leq\frac1{N_k},
\qquad
\frac{(q_k-1)(m-1)}{N_{k+1}-m+1}\leq1.
\]
Therefore
\begin{equation}\label{eq:boundary-total}
\begin{split}
\text{\rm boundary contribution to }Q_{k+1,m}(\mathcal E)
\leq\;&
\frac{C_{\mathrm{pr}}}{N_k}
\sum_{\ell=1}^{m-1}
\bigl(B_{k-1}p^{-\beta_{k-1}\ell}+\varepsilon_k\bigr)
\\
&\hspace{8mm}\cdot
\bigl(B_{k-1}p^{-\beta_{k-1}(m-\ell)}
+\varepsilon_k\bigr)
+
C_{\mathrm{pr}}\frac{N_{k+1}}{M_{k+1}}.
\end{split}
\end{equation}

Since $\beta_{k-1}>t_k>s$,
\[
\frac1{N_k}\sum_{\ell=1}^{m-1}
p^{-\beta_{k-1}\ell}
p^{-\beta_{k-1}(m-\ell)}
=
\frac{m-1}{N_k}p^{-\beta_{k-1}m}
\leq
p^{-\beta_{k-1}m},
\]
and
\[
\sum_{\ell=1}^{m-1}p^{-\beta_{k-1}\ell}\leq \sum_{j=1}^{\infty}p^{-sj}
=\frac{1}{p^s-1}.
\]
Expanding the product in \eqref{eq:boundary-total} therefore gives
\[
\begin{split}
\text{\rm boundary contribution}
\leq C_{\mathrm{pr}}\Bigl(
&B_{k-1}^2p^{-\beta_{k-1}m}
+
\frac{2}{p^s-1}\frac{B_{k-1}}{N_k}\varepsilon_k
\\
&+
\varepsilon_k^2
+
\frac{N_{k+1}}{M_{k+1}}
\Bigr).
\end{split}
\]

We estimate these four terms separately. Since
$B_{k-1}
=
C_{\mathrm{pr}}A_{k-1}^2
\leq
C_{\mathrm{pr}}C_*^2N_{k-1}^2$
and $\beta_{k-1}>t_k$, while $m>N_{k-1}$,
\[
\begin{split}
B_{k-1}^2p^{-\beta_{k-1}m}
&\leq
C_{\mathrm{pr}}^2C_*^4N_{k-1}^4
p^{-(t_k-s)N_{k-1}}p^{-sm}
\\
&\leq
p^{-sm}
\end{split}
\]
by \eqref{eq:aux-absorption} with $k-1$ in place of $k$.

Moreover, the condition $N_k\geq A_{k-1}^2$ imposed in Step~1 gives
$\frac{B_{k-1}}{N_k}
=
C_{\mathrm{pr}}\frac{A_{k-1}^2}{N_k}
\leq C_{\mathrm{pr}}.$
As in the internal-window estimate, \eqref{eq:aux-epsilon} and
\eqref{eq:aux-absorption} give
$\varepsilon_k\leq p^{-sN_k}\leq p^{-sm}.$
In particular, $\varepsilon_k^2\leq p^{-sm}$. Finally, by the
definition of $\varepsilon_{k+1}$,
$C_{\mathrm{pr}}\frac{N_{k+1}}{M_{k+1}}
=
\varepsilon_{k+1}
\leq
p^{-sN_{k+1}}
\leq
p^{-sm}.$

Thus, setting
$C_{\mathrm{bd}}
:=
1+2C_{\mathrm{pr}}+\frac{2}{p^s-1}C_{\mathrm{pr}}^2,$
we obtain
\begin{equation}\label{eq:boundary-final}
\text{\rm boundary contribution to }Q_{k+1,m}(\mathcal E)
\leq
C_{\mathrm{bd}}p^{-sm}.
\end{equation}
$$ $$

Combining \eqref{eq:internal-bound} and
\eqref{eq:boundary-final},
$Q_{k+1,m}(\mathcal E)
\leq
(2+C_{\mathrm{bd}})p^{-sm}.$ We can now 
define the constant from the statement of the Theorem
\[
C_{\mathrm{fus}}
:=
2+C_{\mathrm{bd}}
=
3+2C_{\mathrm{pr}}+\frac{2}{p^s-1}C_{\mathrm{pr}}^2.
\]
Since $e\in S^1$, $I\subset\mathbb R$, and
$N_{k-1}<m\leq N_k$ were arbitrary, this proves
\eqref{eq:fusion-projections} with constant $C_{\mathrm{fus}}$.

Finally, $C_{\mathrm{fus}}\geq2$, so
\eqref{eq:fusion-frequency-bound} from Step~3 also gives
\eqref{eq:fusion-frequencies} with the same constant
$C_{\mathrm{fus}}$. Thus, taking the constant in
Theorem~\ref{thm:fusion} to be $C_{\mathrm{fus}}$, all the stated
properties hold, and the proof is complete.
\end{proof}

\subsection{Tube occupancy and entropy of $\nu$}\label{subsec:projection-entropy}
Recall that, for $\omega=(\omega_n)_{n\geq1}\in\mathcal A_p^{\mathbb N}$,
\[
\Pi_m(\omega):=\sum_{n=1}^m p^{-n}\omega_n
\]
denotes the $m$-digit truncation of $\Pi$.
\begin{cor}\label{cor:fusion-entropy}
Let $\nu$ be the invariant measure given by
Theorem~\ref{thm:fusion}, and set
$C_{\mathrm{occ}}
:=
\max\{2C_{\mathrm{fus}},p^{sN_1}\}.$
Then
\begin{equation}\label{eq:nu-entropy}
h_\nu(\sigma)=s\log p.
\end{equation}
Moreover, for every $m\geq1$, every $e\in S^1$, and every interval
$I\subset\mathbb R$ with $|I|\leq C_0p^{-m}$,
\begin{equation}\label{eq:truncated-occupancy}
\nu\left(
\left\{
\omega:
e\cdot\Pi_m(\omega)\in I
\right\}
\right)
\leq
C_{\mathrm{occ}}p^{-sm}.
\end{equation}
Therefore, for every $e\in S^1$ and $m\geq1$,
\begin{equation}\label{eq:projected-entropy}
\mathrm H\left(
(P_e\circ\Pi_m)_*\nu,\mathcal D_m
\right)
\geq
sm\log p-\log C_{\mathrm{occ}}.
\end{equation}
\end{cor}

\begin{proof}
We first prove \eqref{eq:truncated-occupancy}. Let $m>N_1$, and choose
$k\geq2$ so that
$N_{k-1}<m\leq N_k.$
Fix $e\in S^1$ and an interval $I\subset\mathbb R$ with
$|I|\leq C_0p^{-m}$, and set
\[
\mathcal E
:=
\left\{
u\in\mathcal A_p^m:
e\cdot X_{0,m}(u)\in I
\right\}.
\]
Since $X_{0,m}(u)=\Pi_m(u)$,
$\nu([\mathcal E])
=
\nu\left(
\left\{
\omega:
e\cdot\Pi_m(\omega)\in I
\right\}
\right).$
By \ref{item:fusion-projections},
$Q_{k+1,m}(\mathcal E)
\leq
C_{\mathrm{fus}}p^{-sm},$
while \ref{item:fusion-frequencies} gives
\[
\left|
\nu([\mathcal E])-Q_{k+1,m}(\mathcal E)
\right|
\leq
C_{\mathrm{fus}}\frac{m}{N_{k+1}}.
\]
Using $m\leq N_k$ and \ref{item:fusion-growth},
$\frac{m}{N_{k+1}}
\leq
\frac{N_k}{N_{k+1}}
\leq
p^{-2N_k}
\leq
p^{-2m}
\leq
p^{-sm}.$
Therefore
\[
\nu([\mathcal E])
\leq
2C_{\mathrm{fus}}p^{-sm}
\leq
C_{\mathrm{occ}}p^{-sm}.
\]

If $1\leq m\leq N_1$, then simply
$\nu([\mathcal E])
\leq1
\leq
p^{sN_1}p^{-sm}
\leq
C_{\mathrm{occ}}p^{-sm}.$
This proves \eqref{eq:truncated-occupancy} for every $m\geq1$.

Since every $J\in\mathcal D_m$ has length $p^{-m}$,
\eqref{eq:truncated-occupancy} gives
\[
\max_{J\in\mathcal D_m}
(P_e\circ\Pi_m)_*\nu(J)
\leq
C_{\mathrm{occ}}p^{-sm}.
\]
Using the elementary bound
$\mathrm H(\theta,\mathcal Q)
\geq
-\log\max_{Q\in\mathcal Q}\theta(Q),$
we obtain
\[
\mathrm H\left(
(P_e\circ\Pi_m)_*\nu,\mathcal D_m
\right)
\geq
sm\log p-\log C_{\mathrm{occ}},
\]
which proves \eqref{eq:projected-entropy}.

It remains to prove \eqref{eq:nu-entropy}. Let $\mathcal P_m$ be the
partition of $\mathcal A_p^{\mathbb N}$ into cylinders of length $m$.
Since $P_e\circ\Pi_m$ is constant on every atom of $\mathcal P_m$,
the partition
$(P_e\circ\Pi_m)^{-1}\mathcal D_m$
is a coarsening of $\mathcal P_m$. Hence
\[
\mathrm H(\nu,\mathcal P_m)
\geq
\mathrm H\left(
\nu,(P_e\circ\Pi_m)^{-1}\mathcal D_m
\right)
=
\mathrm H\left(
(P_e\circ\Pi_m)_*\nu,\mathcal D_m
\right).
\]
By \eqref{eq:projected-entropy},
$\mathrm H(\nu,\mathcal P_m)
\geq
sm\log p-\log C_{\mathrm{occ}}.$
Dividing by $m$ and letting $m\to\infty$ gives
$h_\nu(\sigma)\geq s\log p.$

For the reverse inequality, fix $k$. Since $X\subset X_k$, every
length-$n$ word occurring in $X$ is contained, up to its initial phase,
in a concatenation of at most
$\left\lceil\frac{n}{N_k}\right\rceil+2$
words from $\mathcal W_k$. Thus, if $p_X(n)$ denotes the number of
length-$n$ words occurring in $X$, then
$p_X(n)
\leq
N_kM_k^{\lceil n/N_k\rceil+2}.$
It follows that
\[
h_\nu(\sigma)
\leq
h_{\mathrm{top}}(X)
\leq
\frac{\log M_k}{N_k}.
\]
By \ref{item:fusion-cardinality},
$\frac{\log M_k}{N_k}
=
t_k\log p+o(1),$
and letting $k\to\infty$ gives
$h_\nu(\sigma)\leq s\log p.$
Combining the two inequalities proves \eqref{eq:nu-entropy}.
\end{proof}

\subsection{Conclusion of the proof}
Let
$\mu:=\Pi_*\nu.$
Since
\[
\Pi\circ\sigma=T_p\circ\Pi
 \text{ modulo }\mathbb Z^2,
\]
the measure $\mu$ is $T_p$-invariant. Since $\nu$ is ergodic, so is
$\mu$.

We first pass \eqref{eq:truncated-occupancy} from the truncated coding
map to the full projections. Recall that
$|\Pi(\omega)-\Pi_m(\omega)|
\leq
\sqrt2\,p^{-m}.$
Hence, uniformly in $e\in S^1$,
$|P_e\Pi(\omega)-P_e\Pi_m(\omega)|
\leq
\sqrt2\,p^{-m}.$
Let $J\in\mathcal D_m$, and let $J^+$ be the
$\sqrt2\,p^{-m}$-neighbourhood of $J$. Then
\[
|J^+|
\leq
(1+2\sqrt2)p^{-m}
<
C_0p^{-m},
\text{ and }
P_e\Pi(\omega)\in J
\Rightarrow
P_e\Pi_m(\omega)\in J^+.
\]
Therefore \eqref{eq:truncated-occupancy} gives
$(P_e\mu)(J)
\leq
C_{\mathrm{occ}}p^{-sm}.$
Thus
\begin{equation}\label{eq:full-projection-occupancy}
\sup_{e\in S^1}
\max_{J\in\mathcal D_m}
(P_e\mu)(J)
\leq
C_{\mathrm{occ}}p^{-sm},\,
\,
m\geq1.
\end{equation}
Set
$C_{\mathrm{tube}}:=2p^sC_{\mathrm{occ}}.$
By Lemma~\ref{lem:padic-to-tubes},
\begin{equation}\label{eq:final-tube-bound}
\mu(T)\leq C_{\mathrm{tube}}w(T)^s
\end{equation}
for every Euclidean tube $T\subset\mathbb R^2$.

It remains to identify the dimension. By the
Shannon--McMillan--Breiman theorem, for $\nu$-almost every $\omega$,
$-\frac1m\log\nu([\omega_1\cdots\omega_m])
\longrightarrow
h_\nu(\sigma)
=
s\log p.$
On the other hand,
$\Pi([\omega_1\cdots\omega_m])
\subset
B\left(\Pi(\omega),\sqrt2\,p^{-m}\right),$
and therefore
\[
\mu\left(
B\left(\Pi(\omega),\sqrt2\,p^{-m}\right)
\right)
\geq
\nu([\omega_1\cdots\omega_m]).
\]
It follows that
$\underline{\dim}_{\mathrm{loc}}
\mu(\Pi(\omega))
\leq s$
for $\nu$-almost every $\omega$, and hence
$\dim\mu\leq s.$

Conversely, every Euclidean ball of radius $r$ is contained in a tube
of width $2r$. Hence \eqref{eq:final-tube-bound} gives
$\mu(B(x,r))
\leq
2^sC_{\mathrm{tube}}r^s$
for every $x$ and $r>0$. Therefore
$\dim\mu\geq s,$
and so
$\dim\mu=s.$

For every $e\in S^1$ and every interval $J\subset\mathbb R$,
\eqref{eq:final-tube-bound} gives
$(P_e\mu)(J)
=
\mu(P_e^{-1}J)
\leq
C_{\mathrm{tube}}|J|^s.$
Hence
$\dim P_e\mu\geq s.$
Since $P_e$ is Lipschitz and $\dim\mu=s$,
$\dim P_e\mu\leq s.$
Thus
\[
\dim P_e\mu=s
\,
\text{for every }e\in S^1.
\]
The same conclusion holds for every non-zero linear map
$L:\mathbb R^2\to\mathbb R$, since such an $L$ is a non-zero scalar
multiple of some $P_e$.

Finally, Theorem~\ref{thm:exact-dimensionality} implies that $\mu$ and
every $P_e\mu$ are exact dimensional. Multiplication by a non-zero
scalar preserves exact dimensionality, so the same is true for every
non-zero linear image $L\mu$.

This completes the proof of Theorem~\ref{thm:main}.

\section{Construction of the initial library and pruning}
\label{sec:library-pruning-proofs}

\subsection{Constructing the initial library: proof of Lemma~\ref{lem:initial-library}}

Recall Definition \ref{Def windown}, and  that
$0<\alpha<\rho<1.$
We seek a constant $A_0\geq1$, depending only on
$p,\alpha,\rho$ and $C_0$, such that, for every sufficiently large
$N$, there exists a library
$\mathcal W\subset\mathcal A_p^N,
\,
|\mathcal W|=\lfloor p^{\rho N}\rfloor,$
for which, simultaneously for every internal window $(a,m)$, every
$e\in S^1$, and every interval $I\subset\mathbb R$ with
$|I|\leq C_0p^{-m}$,
\[
\frac{1}{|\mathcal W|}
\#\left\{
w\in\mathcal W:
e\cdot X_{a,m}(w)\in I
\right\}
\leq
A_0p^{-\alpha m}.
\]
We construct such a library by sampling uniformly without replacement
from $\mathcal A_p^N$.
\medskip

\begin{proof}[Proof of Lemma~\ref{lem:initial-library}]
Put
$M:=\lfloor p^{\rho N}\rfloor.$
We choose $\mathcal W$ uniformly at random among the $M$-element subsets
of $\mathcal A_p^N$ and show that, with positive probability, it has the
required property.

We first consider one fixed constraint. Fix an internal window $(a,m)$,
a direction $e\in S^1$, and an interval $I$ of length at most
$(C_0+1)p^{-m}.$
If $w$ is chosen uniformly from $\mathcal A_p^N$, then for every
$u\in\mathcal A_p^m$,
\[
\mathbb P\bigl(
(w_{a+1},\ldots,w_{a+m})=u
\bigr)
=
p^{-2m},
\text{ where } \mathbb P \text{ is the uniform measure}.
\]
Thus the $m$-block $(w_{a+1},\ldots,w_{a+m})$ is also uniformly distributed
on $\mathcal A_p^m$. Hence $X_{a,m}(w)$ is uniform on the
$p^{-m}$-grid
\[
\left\{
\sum_{r=1}^m p^{-r}u_r:
u_1,\ldots,u_m\in\mathcal A_p
\right\}
\subset[0,1]^2,
\]
which consists of $p^{2m}$ points.

The condition
$e\cdot X_{a,m}(w)\in I$
restricts the $p^{-m}$-grid to a strip. Since one of the two
coordinates of $e$ has absolute value at least $1/\sqrt2$, fixing the
other coordinate leaves at most
$K_0
:=
\left\lceil \sqrt2(C_0+1)\right\rceil+1$
possible values of that coordinate. Hence the strip contains at most
$K_0p^m$ grid points, uniformly in $e$, $m$, and the position of $I$.
Consequently,
\begin{equation}\label{eq:initial-single-probability}
\mathbb P\left(
e\cdot X_{a,m}(w)\in I
\right)
\leq
K_0p^{-m}.
\end{equation}
Set
\[
S
:=
\left\{
w\in\mathcal A_p^N:
e\cdot X_{a,m}(w)\in I
\right\},\, \text{ and }
Z
:=
|\mathcal W\cap S|
=
\#\left\{
w\in\mathcal W:
e\cdot X_{a,m}(w)\in I
\right\}.
\]
Since $\mathcal W$ is chosen uniformly among the $M$-element subsets
of $\mathcal A_p^N$, every $w\in\mathcal A_p^N$ satisfies
$\mathbb P(w\in\mathcal W)
=
\frac{M}{|\mathcal A_p^N|}.$
Therefore, by \eqref{eq:initial-single-probability},
$\mathbb EZ
=
M\frac{|S|}{|\mathcal A_p^N|}
\leq
K_0Mp^{-m}.$
Choose
\[
A_0\geq4K_0,\, \text{ and set }
T_m:=A_0Mp^{-\alpha m}.
\]
Since $\alpha<1$,
$T_m\geq4\mathbb EZ.$
Also $0\leq Z\leq M$.

Now, if $T_m\geq M$, then
$Z\leq T_m$ automatically, and hence
$\frac{Z}{M}\leq A_0p^{-\alpha m}.$
If $T_m<M$, then, since $Z$ is hypergeometric and
$T_m\geq4\mathbb EZ$, the standard Chernoff bound gives
\[
\mathbb P(Z>T_m)
\leq
e^{-T_m/4}.
\]
For all sufficiently large $N$,
$M=\lfloor p^{\rho N}\rfloor
\geq
\frac12p^{\rho N}.$
Since $m\leq N$,
\[
T_m
=
A_0Mp^{-\alpha m}
\geq
\frac{A_0}{2}p^{\rho N-\alpha m}
\geq
\frac{A_0}{2}p^{(\rho-\alpha)N}.
\]
Therefore
\begin{equation}\label{eq:initial-failure}
\mathbb P\left(
Z>A_0Mp^{-\alpha m}
\right)
\leq
\exp\left(
-\frac{A_0}{8}p^{(\rho-\alpha)N}
\right).
\end{equation}

It remains to make the estimate simultaneous over all constraints. Set
$\eta:=\frac{1}{4(1+\sqrt2)},$
so that
$(2\sqrt2+2)\eta<1.$
For each $m$, choose an $\eta p^{-m}$-net
\[
\mathcal E_m\subset S^1,\, \text{ with }
|\mathcal E_m|
\leq
K_Ep^m,
\qquad
K_E:=1+\frac{2\pi}{\eta}.
\]
We also discretize the position of the interval. If
$I=[t,t+L]$, with $L\leq C_0p^{-m}$, and
$I\cap e\cdot[0,1]^2=\varnothing$, then the corresponding constraint is
empty. Otherwise,
$t\in[-C_0-2,2].$
Thus we may choose an $\eta p^{-m}$-net
\[
\mathcal T_m\subset[-C_0-2,2], \text{ with }
|\mathcal T_m|
\leq
K_Tp^m,
\qquad
K_T:=1+\frac{C_0+4}{\eta}.
\]

We claim that it is enough to impose the concentration estimate for
directions $e'\in\mathcal E_m$ and intervals
\[
I'
:=
\left[
t'-(1+\sqrt2)\eta p^{-m},
\,
t'+C_0p^{-m}+(1+\sqrt2)\eta p^{-m}
\right],
\,
t'\in\mathcal T_m.
\]
Indeed, let $e\in S^1$ and
$I=[t,t+L],
\,
L\leq C_0p^{-m},$
and assume that $I\cap e\cdot[0,1]^2\neq\varnothing$. Choose
$e'\in\mathcal E_m$ and $t'\in\mathcal T_m$ such that
\[
|e-e'|\leq\eta p^{-m},
\,
|t-t'|\leq\eta p^{-m}.
\]
For every $x\in[0,1]^2$,
$|(e-e')\cdot x|
\leq
\sqrt2\,\eta p^{-m}.$
Hence
$e\cdot x\in I
\Rightarrow
e'\cdot x\in I'.$
Moreover,
\[
|I'|
=
\bigl(C_0+(2\sqrt2+2)\eta\bigr)p^{-m}
<
(C_0+1)p^{-m}.
\]
Thus every constraint with arbitrary $e\in S^1$ and
$|I|\leq C_0p^{-m}$ is dominated by one of these discretized
constraints, to which \eqref{eq:initial-failure} applies.

For each $m$, a discretized constraint is specified by
\[
(a,e',t')
\in
\{0,\ldots,N-m\}\times\mathcal E_m\times\mathcal T_m.
\]
Hence the number of discretized constraints over all scales is at most
\[
K_EK_T
\sum_{m=1}^N
(N-m+1)p^{2m}.
\]
Writing $j=N-m$, we have
\[
\sum_{m=1}^N
(N-m+1)p^{2m}
=
p^{2N}
\sum_{j=0}^{N-1}(j+1)p^{-2j}
\leq
\frac{p^{2N}}{(1-p^{-2})^2}.
\]
Therefore, setting
$K_{\mathrm{net}}
:=
\frac{K_EK_T}{(1-p^{-2})^2},$
the total number of discretized constraints is at most
$K_{\mathrm{net}}p^{2N}.$

For each fixed discretized constraint,
\eqref{eq:initial-failure} gives
\[
\mathbb P(\text{that constraint fails})
\leq
\exp\left(
-\frac{A_0}{8}p^{(\rho-\alpha)N}
\right).
\]
Therefore, by the union bound,
\[
\mathbb P(\text{at least one discretized constraint fails})
\leq
K_{\mathrm{net}}p^{2N}
\exp\left(
-\frac{A_0}{8}p^{(\rho-\alpha)N}
\right).
\]
Since $\rho>\alpha$, the right-hand side tends to $0$ as
$N\to\infty$. In particular, for every sufficiently large $N$, there
is a choice of $\mathcal W$ for which none of the discretized
constraints fails.

For such a choice, the domination above gives, simultaneously for
every internal window $(a,m)$, every $e\in S^1$, and every interval
$I\subset\mathbb R$ with $|I|\leq C_0p^{-m}$,
\[
\frac1M
\#\left\{
w\in\mathcal W:
e\cdot X_{a,m}(w)\in I
\right\}
\leq
A_0p^{-\alpha m}.
\]
Thus $\mathcal W$ is $(\alpha,A_0)$-transverse, completing the proof.
\end{proof}
\subsection{Proof of the pruning theorem}
\label{subsec:pruning-proof}
We now prove Theorem~\ref{thm:pruning}. Fix a constant
$\Gamma\geq2,$
depending only on $p$ and $C_0$; 
the precise lower bound required for $\Gamma$ will be specified in \eqref{eq:Gamma-choice}.

Let
$\mathcal W\subset\mathcal A_p^N,
\,
M:=|\mathcal W|,$
be $(\alpha,A)$-transverse, and suppose
\begin{equation}\label{eq:pruning-proof-budget}
0<t<
\alpha-\frac{\log_p(\Gamma A)}{N}.
\end{equation}
For a sufficiently large multiple $q$ of $M$, we put
$N':=qN,
\,
K:=\lfloor p^{tN'}\rfloor.$

We  construct a library
$\mathcal W'\subset\mathcal A_p^{N'},
\,
|\mathcal W'|=K,$
satisfying Theorem~\ref{thm:pruning}\textup{(a)}: every
$W\in\mathcal W'$ is a concatenation of $q$ words from $\mathcal W$,
and every $w\in\mathcal W$ occurs exactly $q/M$ times in $W$.

We  then prove the refined all-window (recall Definition \ref{Def windown}) estimate from
Theorem~\ref{thm:pruning}\textup{(b)}, namely that for some
$\beta>t$,
\[
\frac{1}{K}
\#\left\{
W\in\mathcal W':
e\cdot X_{a,m}(W)\in I
\right\}
\leq
C\left(
A^2p^{-\beta m}
+
\frac{N'}{K}
\right),
\]
uniformly in the range specified in
\eqref{eq:refined-pruning-main}.

Finally, assuming the refined estimate
\eqref{eq:input-refined} for the old library, we  prove
Theorem~\ref{thm:pruning}\textup{(c)}, namely the one-boundary estimate
\[
\frac{1}{K}
\#\left\{
W\in\mathcal W':
e\cdot X_{rN-\ell,m}(W)\in I
\right\}
\leq
C\left[
\bigl(Bp^{-\gamma\ell}+\varepsilon\bigr)
\bigl(Bp^{-\gamma(m-\ell)}+\varepsilon\bigr)
+
\frac{N'}{K}
\right],
\]
uniformly in the range specified in \eqref{eq:boundary-pruning}.

\begin{proof}[Proof of Theorem~\ref{thm:pruning}]
We divide the proof into five steps.

\medskip
\noindent
\textbf{Step 1: the balanced type class.}

Write
$r:=\frac{q}{M},$
and let $\Omega_q$ be the set of all ordered $q$-tuples
$(w_1,\ldots,w_q)\in\mathcal W^q$
in which every $w\in\mathcal W$ occurs exactly $r$ times. We identify
such a tuple with the concatenated word
$w_1\cdots w_q\in\mathcal A_p^{qN}.$
This identification is injective, and hence
$|\Omega_q|=\frac{q!}{(r!)^M}.$
By the standard method-of-types bound,
\begin{equation}\label{eq:type-probability}
|\Omega_q|
\geq
\frac{M^q}{(q+1)^M}.
\end{equation}
Indeed, the $M^q$ elements of $\mathcal W^q$ split into at most
$(q+1)^M$ type classes, and, since $q$ is divisible by $M$, the
balanced class $\Omega_q$ has maximal cardinality.

We next check that $\Omega_q$ contains enough words to prune down to
cardinality
$K=\lfloor p^{tN'}\rfloor.$
Set
\[
\rho:=\frac{\log_p M}{N}.
\]
Since $\mathcal W$ is $(\alpha,A)$-transverse, for every $e\in S^1$
and every interval $I\subset\mathbb R$ with
$|I|\leq C_0p^{-N}$,
\[
\frac1M
\#\left\{
u\in\mathcal W:
e\cdot X_{0,N}(u)\in I
\right\}
\leq
Ap^{-\alpha N}.
\]
Fix $w\in\mathcal W$ and $e\in S^1$, and choose an interval $I$ of
length at most $C_0p^{-N}$ containing $e\cdot X_{0,N}(w)$. Then the
set on the left contains $w$, so
$\frac1M\leq Ap^{-\alpha N}.$
Therefore
\begin{equation}\label{eq:rho-lower}
\rho
\geq
\alpha-\frac{\log_p A}{N}.
\end{equation}
In particular, setting
\[
\alpha_*
:=
\alpha-\frac{\log_p(\Gamma A)}{N},
\]
we have
$\rho
\geq
\alpha_*+\frac{\log_p\Gamma}{N}
>
\alpha_*
>
t,$
where the last inequality is \eqref{eq:pruning-proof-budget}.

Using \eqref{eq:type-probability},
$\frac{1}{qN}\log_p|\Omega_q|
\geq
\rho-\frac{M\log_p(q+1)}{qN}.$
Since $M$ and $N$ are fixed at this stage and
\[
\frac{M\log_p(q+1)}{qN}\rightarrow0
\text{ as } q\to\infty,
\]
while $\rho>t$, it follows that
\[
|\Omega_q|
\geq
p^{tqN}
\geq
\lfloor p^{tN'}\rfloor
=
K
\]
for every sufficiently large multiple $q$ of $M$.
\medskip
\medskip

\noindent
\textbf{Step 2: a deterministic estimate for one window.}

Set
$J:=
\left\lceil
\frac{C_0+\sqrt2}{C_0}
\right\rceil.$
Thus, for every $\ell\geq1$, every interval of length at most
$(C_0+\sqrt2)p^{-\ell}$ can be covered by at most $J$ intervals of
length $C_0p^{-\ell}$.

\begin{lemma}\label{lem:admissible-tuples}
Let
$V=v_1\cdots v_q,
\,
v_j\in\mathcal W,$
and fix a window $(a,m)$ of $V$. Write
$a=(j_0-1)N+b,
\,
0\leq b<N,$
and let $R$ be the number of level-$N$ blocks of $V$ intersected by
this window (these blocks are
$v_{j_0},\ldots,v_{j_0+R-1}$). 
Then
\begin{equation}\label{eq:number-old-blocks}
R\leq\frac{m}{N}+2.
\end{equation}
Moreover, for every $e\in S^1$ and every interval
$I\subset\mathbb R$ with $|I|\leq C_0p^{-m}$,
\begin{equation}\label{eq:admissible-tuples}
\#\left\{
(w_1,\ldots,w_R)\in\mathcal W^R:
e\cdot X_{b,m}(w_1\cdots w_R)\in I
\right\}
\leq
(JAM)^R p^{-\alpha m}.
\end{equation}
\end{lemma}

Thus, for a fixed window $(a,m)$, $e\in S^1$, and
$|I|\leq C_0p^{-m}$, at most
$(JAM)^R p^{-\alpha m}$
ordered $R$-tuples $(w_1,\ldots,w_R)\in\mathcal W^R$ can occur in the
$R$ level-$N$ blocks intersected by the window while satisfying
$e\cdot X_{a,m}(V)\in I.$

\begin{proof}
Fix $(w_1,\ldots,w_R)\in\mathcal W^R$, and let $\ell_j$ be the number
of digits of the window $X_{b,m}(w_1\cdots w_R)$ lying in $w_j$.
Then
$\ell_1+\cdots+\ell_R=m.$
The first piece begins at offset $b$ in $w_1$, the last piece is a
prefix of $w_R$, and every intermediate piece has length $N$. Hence
$m\geq(R-2)N,$
which proves \eqref{eq:number-old-blocks}.

Set
$b_1:=b,
\,
b_j:=0, \, 2\leq j\leq R,
\,
Y_j:=X_{b_j,\ell_j}(w_j).$ 
Then, by  definition,
\begin{equation}\label{eq:window-decomposition}
X_{b,m}(w_1\cdots w_R)
=
Y_1
+p^{-\ell_1}Y_2
+\cdots+
p^{-(\ell_1+\cdots+\ell_{R-1})}Y_R.
\end{equation}
We first count the possible choices of $w_1$. The sum of the terms in
\eqref{eq:window-decomposition} following $Y_1$ has the form
$p^{-\ell_1}z,
\,
z\in[0,1]^2.$
Its projection under $x\mapsto e\cdot x$ therefore has diameter at most
$\sqrt2\,p^{-\ell_1}.$
So, if
$e\cdot X_{b,m}(w_1\cdots w_R)\in I,$
then
$e\cdot Y_1
=
e\cdot X_{b,\ell_1}(w_1)$
belongs to an interval of length at most
\[
|I|+\sqrt2\,p^{-\ell_1}
\leq
C_0p^{-m}+\sqrt2\,p^{-\ell_1}
\leq
(C_0+\sqrt2)p^{-\ell_1},
\]
where we used $\ell_1\leq m$. By the definition of $J$, this interval
can be covered by at most $J$ intervals of length
$C_0p^{-\ell_1}$. Since $\mathcal W$ is $(\alpha,A)$-transverse,
\eqref{eq:transverse-main} applied to the window $(b,\ell_1)$ shows
that there are at most
\[
JAMp^{-\alpha\ell_1}
\]
possible choices for $w_1$.

Fix one such choice of $w_1$. Subtracting $Y_1$ from
\eqref{eq:window-decomposition} and multiplying by $p^{\ell_1}$, the
remaining condition takes the same form:
\[
e\cdot
\left(
Y_2
+p^{-\ell_2}Y_3
+\cdots+
p^{-(\ell_2+\cdots+\ell_{R-1})}Y_R
\right)
\in I_1,
\]
for some interval $I_1$ with
$|I_1|
\leq
C_0p^{-(m-\ell_1)}.$
Applying the preceding argument to the first term
$Y_2=X_{0,\ell_2}(w_2)$
gives at most
$JAMp^{-\alpha\ell_2}$
possible choices for $w_2$.
Continuing successively through the $R$ pieces, we obtain at most
\[
\prod_{j=1}^R
\left(JAMp^{-\alpha\ell_j}\right)
=
(JAM)^R
p^{-\alpha(\ell_1+\cdots+\ell_R)}
=
(JAM)^R p^{-\alpha m}
\]
possible ordered $R$-tuples $(w_1,\ldots,w_R)$. This proves
\eqref{eq:admissible-tuples}.
\end{proof}

We now specify the constant $\Gamma$ announced at the beginning of the
proof. Recall that $J$ depends only on $C_0$. Fix
\begin{equation}\label{eq:Gamma-choice}
\Gamma\geq 2J.
\end{equation}
Then $\Gamma$ depends only on $p$ and $C_0$, and, since $A\geq1$,
also $\Gamma A\geq1$. The choice $\Gamma\geq2J$ will absorb the
additional factor $2^R$ arising in the next step when
\eqref{eq:admissible-tuples} is transferred to the uniform measure on
balanced concatenations.

\medskip

\noindent
\textbf{Step 3: non-concentration in the balanced population.}
Let $V$ be uniformly distributed on $\Omega_q$, viewed as a
concatenated word in $\mathcal A_p^{qN}$, and denote its law by
$\mathbb P_{\Omega_q}$. We show that, for every fixed window $(a,m)$,
every $e\in S^1$, and every interval $I\subset\mathbb R$ with
$|I|\leq C_0p^{-m}$,
\begin{equation}\label{eq:balanced-population}
\mathbb P_{\Omega_q}
\left(
e\cdot X_{a,m}(V)\in I
\right)
\leq
\Gamma^2 A^2p^{-\beta m}
\end{equation}
for some $\beta>t$ independent of $q$, once $q$ is a sufficiently
large multiple of $M$.

Set
$\beta
:=
\frac12(t+\alpha_*),$ where we recall that
$\alpha_*
:=
\alpha-\frac{\log_p(\Gamma A)}{N}.$
By \eqref{eq:pruning-proof-budget},
$t<\beta<\alpha_*.$

Let $R$ be the number of level-$N$ blocks intersected by $(a,m)$, as
in Lemma~\ref{lem:admissible-tuples}. Recall that
\begin{equation}\label{eq:R-recall}
R\leq\frac{m}{N}+2.
\end{equation}
We distinguish two cases.

\medskip

\emph{Case 1: $R\leq q/2$.}

Fix an ordered $R$-tuple
$(u_1,\ldots,u_R)\in\mathcal W^R.$
We estimate the probability that the $R$ level-$N$ blocks intersected
by $(a,m)$ are, in order, $u_1,\ldots,u_R$.

Recall that every element of $\Omega_q$ contains each
$u\in\mathcal W$ exactly
$r:=\frac qM$
times. Suppose that the first $j-1$ intersected blocks have been
prescribed as $u_1,\ldots,u_{j-1}$. If this conditioning event has
positive probability, then among the remaining $q-j+1$ coordinates
at most $r$ can be equal to $u_j$. Hence
\[
\mathbb P_{\Omega_q}
\left(
\begin{array}{c}
\text{the $j$th intersected block is }u_j
\end{array}
\;\middle|\;
\begin{array}{c}
\text{the preceding ones are}\\
u_1,\ldots,u_{j-1}
\end{array}
\right)
\leq
\frac{r}{q-j+1}.
\]
Since $j\leq R\leq q/2$,
$\frac{r}{q-j+1}
\leq
\frac{q/M}{q/2}
=
\frac2M.$
Multiplying the successive conditional probabilities gives
\begin{equation}\label{eq:balanced-tuple-probability}
\mathbb P_{\Omega_q}
\left(
\text{the $R$ intersected blocks are }
u_1,\ldots,u_R
\right)
\leq
\left(\frac2M\right)^R.
\end{equation}

By Lemma~\ref{lem:admissible-tuples}, for the fixed $(a,m)$, $e$, and
$I$, at most
$(JAM)^R p^{-\alpha m}$
ordered $R$-tuples can give rise to the event
$e\cdot X_{a,m}(V)\in I$. Summing
\eqref{eq:balanced-tuple-probability} over these tuples gives
\[
\mathbb P_{\Omega_q}
\left(
e\cdot X_{a,m}(V)\in I
\right)
\leq
(JAM)^Rp^{-\alpha m}
\left(\frac2M\right)^R
=
(2JA)^Rp^{-\alpha m}.
\]
By \eqref{eq:Gamma-choice},
$2JA\leq\Gamma A,$
and, since $A\geq1$, we have $\Gamma A\geq1$. Hence
\eqref{eq:R-recall} gives
$(2JA)^R
\leq
(\Gamma A)^R
\leq
(\Gamma A)^{m/N+2}.$
Therefore
\[
\begin{split}
\mathbb P_{\Omega_q}
\left(
e\cdot X_{a,m}(V)\in I
\right)
&\leq
(\Gamma A)^2
(\Gamma A)^{m/N}p^{-\alpha m}
\\
&=
(\Gamma A)^2p^{-\alpha_*m}
\\
&\leq
\Gamma^2A^2p^{-\beta m},
\end{split}
\]
where the last inequality uses $\beta<\alpha_*$.

\medskip

\emph{Case 2: $R>q/2$.}

Here the preceding conditional-probability estimate is no longer
useful, because the number of unprescribed coordinates may become
small. Instead we use the lower bound for $|\Omega_q|$ from Step~1.

Fix again an ordered $R$-tuple
$(u_1,\ldots,u_R)\in\mathcal W^R.$
Prescribing the $R$ blocks intersected by $(a,m)$ fixes $R$
coordinates of an element of $\Omega_q$. The remaining $q-R$
coordinates each have at most $M$ possible values, so at most
$M^{q-R}$ elements of $\Omega_q$ extend the prescribed tuple.
Consequently,
\[
\mathbb P_{\Omega_q}
\left(
\text{the $R$ intersected blocks are }
u_1,\ldots,u_R
\right)
\leq
\frac{M^{q-R}}{|\Omega_q|}.
\]
By \eqref{eq:type-probability},
$|\Omega_q|
\geq
\frac{M^q}{(q+1)^M},$
and therefore
\begin{equation}\label{eq:long-balanced-tuple-probability}
\mathbb P_{\Omega_q}
\left(
\text{the $R$ intersected blocks are }
u_1,\ldots,u_R
\right)
\leq
M^{-R}(q+1)^M.
\end{equation}
Again, Lemma~\ref{lem:admissible-tuples} shows that at most
$(JAM)^R p^{-\alpha m}$
ordered $R$-tuples can give rise to
$e\cdot X_{a,m}(V)\in I$. Summing
\eqref{eq:long-balanced-tuple-probability} over them gives
\[
\mathbb P_{\Omega_q}
\left(
e\cdot X_{a,m}(V)\in I
\right)
\leq
(JA)^Rp^{-\alpha m}(q+1)^M.
\]
Since $J\leq\Gamma$ and $\Gamma A\geq1$,
$(JA)^R
\leq
(\Gamma A)^R
\leq
(\Gamma A)^{m/N+2},$
and hence
\begin{equation}\label{eq:long-window-preliminary}
\mathbb P_{\Omega_q}
\left(
e\cdot X_{a,m}(V)\in I
\right)
\leq
(\Gamma A)^2p^{-\alpha_*m}(q+1)^M.
\end{equation}

It remains to absorb the factor $(q+1)^M$. Since $R>q/2$ and
$m\geq(R-2)N$, for $q\geq8$,
\[
m
>
\left(\frac q2-2\right)N
\geq
\frac{qN}{4}.
\]
Thus
$(q+1)^M
=
p^{M\log_p(q+1)}
\leq
p^{\eta_qm}$. Define
$\eta_q
:=
\frac{4M\log_p(q+1)}{qN}.$
Here $M$ and $N$ are fixed while $q\to\infty$, so
$\eta_q\rightarrow0.$
Choose $q$, still a multiple of $M$, sufficiently large that
$q\geq8$ and 
$\eta_q\leq\alpha_*-\beta.$ 
Then \eqref{eq:long-window-preliminary} gives
\[
\begin{split}
\mathbb P_{\Omega_q}
\left(
e\cdot X_{a,m}(V)\in I
\right)
&\leq
(\Gamma A)^2
p^{-(\alpha_*-\eta_q)m}
\\
&\leq
\Gamma^2A^2p^{-\beta m}.
\end{split}
\]

Combining the two cases proves \eqref{eq:balanced-population},
uniformly over all windows $(a,m)$, directions $e\in S^1$, and
intervals $I$ with $|I|\leq C_0p^{-m}$.

\medskip

\noindent
\textbf{Step 4: the one-boundary estimate in the balanced population.}
Let $\gamma>0$ and $\varepsilon\geq0$, and suppose that, for some
$B\geq1$, the library $\mathcal W$ satisfies
\eqref{eq:input-refined}. Thus, for every internal window $(a,n)$,
every $e\in S^1$, and every interval $H\subset\mathbb R$ with
$|H|\leq C_0p^{-n}$,
\begin{equation}\label{eq:input-refined-recalled}
\frac1M
\#\left\{
w\in\mathcal W:
e\cdot X_{a,n}(w)\in H
\right\}
\leq
Bp^{-\gamma n}+\varepsilon.
\end{equation}

We claim that, for every
$1\leq j<q,
\,
1\leq\ell<m\leq N,$
every $e\in S^1$, and every interval $I\subset\mathbb R$ with
$|I|\leq C_0p^{-m}$,
\begin{equation}\label{eq:boundary-population}
\mathbb P_{\Omega_q}
\left(
e\cdot X_{jN-\ell,m}(V)\in I
\right)
\leq
2J
\bigl(Bp^{-\gamma\ell}+\varepsilon\bigr)
\bigl(Bp^{-\gamma(m-\ell)}+\varepsilon\bigr).
\end{equation}

Fix $j,\ell,m,e$, and $I$ as above. The window
$(jN-\ell,m)$ consists of the last $\ell$ digits of $v_j$ followed by
the first $m-\ell$ digits of $v_{j+1}$. Writing
$u:=v_j,
\,
v:=v_{j+1},$
we have
\begin{equation}\label{eq:boundary-decomposition}
X_{jN-\ell,m}(V)
=
X_{N-\ell,\ell}(u)
+
p^{-\ell}X_{0,m-\ell}(v).
\end{equation}

Set
$b_1:=Bp^{-\gamma\ell}+\varepsilon,
\,
b_2:=Bp^{-\gamma(m-\ell)}+\varepsilon.$
We first count the possible choices of $u$. Since
$X_{0,m-\ell}(v)\in[0,1]^2,$
the projection of the second term in
\eqref{eq:boundary-decomposition} has diameter at most
$\sqrt2\,p^{-\ell}.$
Hence the condition
$e\cdot X_{jN-\ell,m}(V)\in I$
forces
$e\cdot X_{N-\ell,\ell}(u)$
to lie in an interval of length at most
\[
|I|+\sqrt2\,p^{-\ell}
\leq
C_0p^{-m}+\sqrt2\,p^{-\ell}
\leq
(C_0+\sqrt2)p^{-\ell}.
\]
By the definition of $J$ in Step~2, this interval can be covered by at
most $J$ intervals of length $C_0p^{-\ell}$. Applying
\eqref{eq:input-refined-recalled} to the internal window
$(N-\ell,\ell)$ shows that there are at most
$JMb_1$
possible choices for $u$.

Fix one such $u$. By \eqref{eq:boundary-decomposition}, the condition
on $v$ is
$e\cdot X_{0,m-\ell}(v)
\in
p^\ell
\left(
I-e\cdot X_{N-\ell,\ell}(u)
\right).$
This interval  has length at most
$p^\ell C_0p^{-m}
=
C_0p^{-(m-\ell)}.$
Therefore \eqref{eq:input-refined-recalled}, applied to the internal
window $(0,m-\ell)$, gives at most
$Mb_2$
possible choices for $v$. Consequently,
\begin{equation}\label{eq:boundary-pair-count}
\#\left\{
(u,v)\in\mathcal W^2:
e\cdot
\left(
X_{N-\ell,\ell}(u)
+
p^{-\ell}X_{0,m-\ell}(v)
\right)
\in I
\right\}
\leq
JM^2b_1b_2.
\end{equation}

It remains to estimate the probability of a prescribed adjacent pair
under $\mathbb P_{\Omega_q}$. Fix $(u,v)\in\mathcal W^2$. By symmetry,
$\mathbb P_{\Omega_q}(v_j=u)=\frac1M.$
Conditional on $v_j=u$, among the remaining $q-1$ coordinates at most
$q/M$ can be equal to $v$. Hence, for $q\geq2$,
\[
\mathbb P_{\Omega_q}
\left(
v_{j+1}=v
\,\middle|\,
v_j=u
\right)
\leq
\frac{q/M}{q-1}
\leq
\frac2M.
\]
Thus
\begin{equation}\label{eq:boundary-pair-probability}
\mathbb P_{\Omega_q}
\left(
v_j=u,\ v_{j+1}=v
\right)
\leq
\frac{2}{M^2}.
\end{equation}

Summing \eqref{eq:boundary-pair-probability} over the pairs counted in
\eqref{eq:boundary-pair-count}, we obtain
\[
\mathbb P_{\Omega_q}
\left(
e\cdot X_{jN-\ell,m}(V)\in I
\right)
\leq
2Jb_1b_2.
\]
Substituting the definitions of $b_1$ and $b_2$ proves
\eqref{eq:boundary-population}.

\medskip

\noindent
\textbf{Step 5: random pruning.}
Choose $\mathcal W'$ uniformly at random among the $K$-element subsets
of $\Omega_q$. By the definition of $\Omega_q$, every
$W\in\mathcal W'$ is a concatenation of $q$ words from $\mathcal W$,
and every $w\in\mathcal W$ occurs exactly $q/M$ times. Thus
Theorem~\ref{thm:pruning}\textup{(a)} holds automatically.

We show that, with positive probability, the population estimates from
Steps~3 and~4 hold simultaneously on $\mathcal W'$, up to the sampling
error introduced by the pruning.

As in the proof of Lemma~\ref{lem:initial-library}, fix
$\eta:=\frac{1}{4(1+\sqrt2)},$
and, for every $1\leq m\leq N'$, choose $\eta p^{-m}$-nets
\[
\mathcal E_m\subset S^1,
\,
\mathcal T_m\subset[-C_0-2,2], \text{ with }
|\mathcal E_m|\leq K_Ep^m,
\,
|\mathcal T_m|\leq K_Tp^m,
\]
where
$K_E:=1+\frac{2\pi}{\eta},
\,
K_T:=1+\frac{C_0+4}{\eta}.$
Using the same enlarged intervals as there, every constraint with
arbitrary $e\in S^1$ and $|I|\leq C_0p^{-m}$ is contained in one of
the corresponding discretized constraints, whose interval has length
less than
$(C_0+1)p^{-m}.$
Since $C_0=10$, every such interval is covered by at most two intervals
of length $C_0p^{-m}$. Thus, the population estimates from
Steps~3 and~4 give, for every discretized constraint $\mathcal C$,
\begin{align}
\mathbb P_{\Omega_q}(\mathcal C)
&\leq
2\Gamma^2A^2p^{-\beta m}
\label{eq:discrete-population-internal}
\end{align}
in the all-window case, and
\begin{align}
\mathbb P_{\Omega_q}(\mathcal C)
&\leq
4J
\bigl(Bp^{-\gamma\ell}+\varepsilon\bigr)
\bigl(Bp^{-\gamma(m-\ell)}+\varepsilon\bigr)
\label{eq:discrete-population-boundary}
\end{align}
in the one-boundary case.

We impose these discretized constraints for the following two
families:
\begin{enumerate}[label=\textup{(\roman*)}]
\item every window $(a,m)$ with
\[
1\leq m\leq N',
\qquad
0\leq a\leq N'-m;
\]

\item every one-boundary window
\[
(jN-\ell,m),
\qquad
1\leq j<q,
\qquad
1\leq\ell<m\leq N.
\]
\end{enumerate}

We first bound the number of constraints. Set
\[
K_{\mathrm{net}}
:=
\frac{K_EK_T}{(1-p^{-2})^2}.
\]
The number of constraints of type~\textup{(i)} is at most
\[
K_EK_T
\sum_{m=1}^{N'}
(N'-m+1)p^{2m}
\leq
K_{\mathrm{net}}p^{2N'}.
\]
Every one-boundary window in \textup{(ii)} is, in particular, a window
of a level-$N'$ word, so the number of constraints of type~\textup{(ii)}
is bounded by the same quantity. Hence the total number of
discretized constraints is at most
\begin{equation}\label{eq:number-pruning-constraints}
2K_{\mathrm{net}}p^{2N'}.
\end{equation}

Fix one discretized constraint $\mathcal C$, and write
\[
\theta_{\mathcal C}
:=
\mathbb P_{\Omega_q}(\mathcal C),\, \text{ and }
Z_{\mathcal C}
:=
\#\left\{
W\in\mathcal W':
W\text{ satisfies }\mathcal C
\right\}.
\]
Since $\mathcal W'$ is chosen uniformly among the $K$-element subsets
of $\Omega_q$, the random variable $Z_{\mathcal C}$ is hypergeometric
with
$\mathbb EZ_{\mathcal C}
=
K\theta_{\mathcal C}.$

The standard Bernstein inequality for hypergeometric random variables
gives, for every $D>0$,
\[
\begin{split}
\mathbb P\left(
Z_{\mathcal C}
>
2K\theta_{\mathcal C}+DN'
\right)
&\leq
\exp\left(
-\frac{(K\theta_{\mathcal C}+DN')^2}
{2\bigl(K\theta_{\mathcal C}
+(K\theta_{\mathcal C}+DN')/3\bigr)}
\right)
\\
&\leq
\exp\left(-\frac{3DN'}8\right).
\end{split}
\]
Choose
$D
:=
1+\frac83
\left(
2\log p+\log(2K_{\mathrm{net}})
\right).$
Then, by \eqref{eq:number-pruning-constraints} and the union bound,
\[
\begin{split}
&\mathbb P\left(
\text{at least one discretized constraint fails}
\right)
\\
&\qquad\leq
2K_{\mathrm{net}}p^{2N'}
\exp\left(-\frac{3DN'}8\right)
<1.
\end{split}
\]
Hence there exists a choice of $\mathcal W'$ for which, simultaneously
for every discretized constraint $\mathcal C$,
\begin{equation}\label{eq:all-pruning-constraints}
Z_{\mathcal C}
\leq
2K\theta_{\mathcal C}+DN'.
\end{equation}
Fix such a choice of $\mathcal W'$.

Now let $(a,m)$ be any window of a level-$N'$ word, let $e\in S^1$,
and let $I\subset\mathbb R$ satisfy $|I|\leq C_0p^{-m}$. The
discretization above places this event inside a constraint
$\mathcal C$ satisfying \eqref{eq:discrete-population-internal}.
Therefore \eqref{eq:all-pruning-constraints} gives
\[
\begin{split}
\frac1K
\#\left\{
W\in\mathcal W':
e\cdot X_{a,m}(W)\in I
\right\}
&\leq
\frac{Z_{\mathcal C}}K
\\
&\leq
4\Gamma^2A^2p^{-\beta m}
+
D\frac{N'}K.
\end{split}
\]

Similarly, assuming \eqref{eq:input-refined}, let
$1\leq j<q,
\,
1\leq\ell<m\leq N,$
and let $e\in S^1$ and $I\subset\mathbb R$ satisfy
$|I|\leq C_0p^{-m}$. The corresponding event is contained in a
discretized one-boundary constraint satisfying
\eqref{eq:discrete-population-boundary}. Hence
\[
\begin{split}
&\frac1K
\#\left\{
W\in\mathcal W':
e\cdot X_{jN-\ell,m}(W)\in I
\right\}
\\
&\qquad\leq
8J
\bigl(Bp^{-\gamma\ell}+\varepsilon\bigr)
\bigl(Bp^{-\gamma(m-\ell)}+\varepsilon\bigr)
+
D\frac{N'}K.
\end{split}
\]

Finally, set
$C
:=
\max\{4\Gamma^2,\,8J,\,D\}.$
Then
\[
\frac1K
\#\left\{
W\in\mathcal W':
e\cdot X_{a,m}(W)\in I
\right\}
\leq
C
\left(
A^2p^{-\beta m}
+
\frac{N'}K
\right),
\]
which is \eqref{eq:refined-pruning-main}, while
\[
\begin{split}
&\frac1K
\#\left\{
W\in\mathcal W':
e\cdot X_{jN-\ell,m}(W)\in I
\right\}
\\
&\qquad\leq
C
\left[
\bigl(Bp^{-\gamma\ell}+\varepsilon\bigr)
\bigl(Bp^{-\gamma(m-\ell)}+\varepsilon\bigr)
+
\frac{N'}K
\right],
\end{split}
\]
which is \eqref{eq:boundary-pruning}. Thus
Theorem~\ref{thm:pruning}\textup{(b)} and~\textup{(c)} hold with the
same constant $C$, while part~\textup{(a)} was built
into the choice $\mathcal W'\subset\Omega_q$. This completes the proof
of Theorem~\ref{thm:pruning}.
\end{proof}

\section{Bilateral determinism and rational dimension drop}\label{sec:positive}
In this section we prove Theorem~\ref{cor:quasi-bernoulli-drop}: A quasi-Bernoulli stationary measure of positive dimension always has a dimension dropping projection. The proof follows by developing a criterion using the notion of bilateral determinism, due to Ornstein and Weiss \cite{OrnsteinWeiss1975}.  We
first recall the precise definitions.

Let $\mu \in \mathcal{P}(\mathbb{T}^2)$ be an ergodic $T_p$-invariant probability measure with
$0<\dim_H\mu=s\leq1.$ Let $\nu$ be the $\sigma$-invariant measure on
$\mathcal A_p^{\mathbb N}$ such that
$\mu=\Pi_*\nu,$
and let $\widetilde\nu$ denote its natural extension to
$\mathcal A_p^{\mathbb Z}$. Then, writing $\sigma=\sigma_p$ for the left shift map on $\mathcal A_p^{\mathbb Z}$, 
\[
h:=h_{\widetilde\nu}(\sigma)
=h_\nu(\sigma)
=h_\mu(T_p)
=s\log p.
\] 
For $m\geq1$, set
\begin{equation}\label{eq:Dmdef2}
D_m
:=
\mathrm H_{\widetilde\nu}
\bigl(
\omega_1^m
\mid
\omega_{\leq0},\omega_{>m}
\bigr),
\qquad
D_0:=0.
\end{equation}
Thus $D_m$ is the uncertainty remaining in a deleted block of length
$m$ after its entire bilateral exterior has been revealed.

Following Ornstein--Weiss, we say that the process is
\emph{bilaterally deterministic} if every finite block is determined,
almost surely, by its bilateral exterior. In the present notation this
is equivalent to
$D_m=0
\text{ for every }m\geq1.$
We emphasize that bilateral determinism here is a property of the
standard $p$-adic symbolic presentation of $\mu$, rather than of the
underlying measure-preserving system alone. Indeed, Ornstein and Weiss
showed that every ergodic transformation with a finite generator admits
a bilaterally deterministic finite generator.

Our main criterion is that failure of bilateral determinism for this
particular coding forces rational dimension drop.

\begin{thm}\label{thm:predictive-drop}
Let $\mu\in\mathcal P(\mathbb T^2)$ be ergodic and $T_p$-invariant, with
$\dim\mu=s>0$. Let $\nu \in\mathcal P(\mathcal A_p^{\mathbb N})$ be $\sigma$-invariant  such that $\mu=\Pi_*\nu$,  and let
$\widetilde\nu\in\mathcal P(\mathcal A_p^{\mathbb Z})$ be the natural
extension of $\nu$. If $\widetilde\nu$ is not bilaterally deterministic,
then there exists 
$a\in\mathbb Z^2\setminus\{0\}$ such that
$\dim (P_a \mu)<s.$
\end{thm}
The main content of Theorem~\ref{thm:predictive-drop} is  in
the range $0<s<1$. For $s>1$ the conclusion is immediate,  while the case
$s=1$ follows from the theorem of Py\"or\"al\"a, Shmerkin, Suomala, and
Wu \cite{Aleksi2025pablo} discussed in the Introduction.

The converse to Theorem~\ref{thm:predictive-drop} is false: bilateral
determinism does not preclude dimension-dropping projections.
Indeed, let $\mu$ be the measure supplied by Theorem~\ref{thm:main}
for $p=2$, with
$\dim\mu=\frac12,$
and let $\widetilde\nu$ be its natural extension on the two-sided coding 
$\{0,1\}^2$. Since every projection of $\mu$ preserves dimension, the
contrapositive of Theorem~\ref{thm:predictive-drop} implies that
$\widetilde\nu$ is bilaterally deterministic.
Now relabel the four symbols by the injective one-block map
\[
\phi(i,j):=(2i+j,0)\in\{0,1,2,3\}^2,
\]
and let $\eta$ be the corresponding $T_4$-invariant measure on
$\mathbb T^2$. If $(X_n)_{n\in\mathbb Z}$ denotes the original
two-sided process and $Y_n:=\phi(X_n)$ the relabelled one, then
injectivity of $\phi$ implies that $X_n$ can be recovered from $Y_n$ at
each coordinate. Hence, for every $m\geq1$,
$\mathrm H(Y_1^m\mid Y_{\leq0},Y_{>m})
=
\mathrm H(X_1^m\mid X_{\leq0},X_{>m})
=
0.$
Thus the natural base-$4$ coding of $\eta$ is again bilaterally
deterministic.

On the other hand, the symbolic entropy is unchanged, and therefore
\[
\dim\eta
=
\frac{h_\nu(\sigma)}{\log4}
=
\frac{\frac12\log2}{\log4}
=
\frac14.
\]
Yet $\eta$ is supported on the horizontal axis. Hence, for
$V=\operatorname{span}\{(0,1)\},$
we have
$\pi_V\eta=\delta_0,$
so the vertical projection strictly drops dimension.

This section is organized as follows. We first show in
Section~\ref{subsec:localdrop} that if
$D_m>0$
for some $m\geq1$, then there exists a primitive
$a\in\mathbb Z^2\setminus\{0\}$ such that
$\dim(P_a\mu)<s.$
Moreover, one may choose
$|a|_\infty\leq p^m-1.$
Next, in Section~\ref{subsec:convexDm}, we show that the increments
$D_m-D_{m-1}$
are nondecreasing. It follows that either $D_m=0$ for every $m\geq1$,
or else $D_m$ grows at least linearly once it becomes positive.

Finally, in Section~\ref{subsec:finiteE}, we relate $D_m$ to the
predictive information of the process; informally, this is the mutual
information between the past and the future. In particular, finite
predictive information implies that $D_m>0$ for some $m$, and hence
forces a dimension-dropping rational projection. We then apply this
criterion to deduce Theorem~\ref{cor:quasi-bernoulli-drop}.

\subsection{Local bilateral non-determinism}\label{subsec:localdrop}

We begin with a local criterion, showing that  positive bilateral
uncertainty at a finite scale already gives a  dimension dropping projection.

\begin{proposition}\label{prop:local-drop}
Suppose that
$D_m>0$
for some $m\geq1$. Then there exists a primitive
$a\in\mathbb Z^2\setminus\{0\}$
with $|a|_\infty\leq p^m-1,$
such that
$\dim(P_a\mu)<s.$
\end{proposition}

\begin{proof}
Fix $m\geq1$ with $D_m>0$.

\medskip
\noindent
\textbf{Step 1: From $D_m>0$ to a one-window collision.}

Consider the $\sigma$-algebra 
$\mathscr G_m
:=
\sigma(\omega_n:n\leq0\text{ or }n>m)$ on $\mathcal{A}_p ^\mathbb{Z}$.
By definition,
\[
D_m
=
\mathrm H_{\widetilde\nu}
\bigl(
\omega_1^m\mid\mathscr G_m
\bigr)
>0.
\]
For each $U\in\mathcal A_p^m$, write
\[
q_U
:=
\widetilde\nu
\bigl(
\omega_1^m=U\mid\mathscr G_m
\bigr).
\]
Since $D_m>0$, the conditional distribution of $\omega_1^m$ given
$\mathscr G_m$ is not a point mass, on a set of positive
$\widetilde\nu$-measure. Hence
\[
\widetilde\nu\left(
\bigcup_{U\neq V}
\{q_U>0,\ q_V>0\}
\right)>0.
\]
There are only finitely many pairs $U\neq V$, so we may fix distinct
$U,V\in\mathcal A_p^m$ such that
\[
\widetilde\nu \left( \{q_U>0,\ q_V>0\}\right)>0.
\]
Moreover,
$\{q_U>0,\ q_V>0\}
=
\bigcup_{\ell=1}^{\infty}
\{q_U\geq\ell^{-1},\ q_V\geq\ell^{-1}\},$
so there exists $\ell_0\geq1$ such that the event
\[
G
:=
\{q_U\geq\ell_0^{-1},\ q_V\geq\ell_0^{-1}\}
\in\mathscr G_m
\]
has positive $\widetilde\nu$-measure. Set
$\delta:=\ell_0^{-1},$ and
$\kappa:=\widetilde\nu(G)>0.$
Then, on $G$,
\begin{equation}\label{eq:twoWordsConditional}
\widetilde\nu(\omega_1^m=U\mid\mathscr G_m)\geq\delta,
\qquad
\widetilde\nu(\omega_1^m=V\mid\mathscr G_m)\geq\delta.
\end{equation}

We now choose a rational direction in which these two words collide.
Set
\[
v
:=
p^m\bigl(\Pi_m(U)-\Pi_m(V)\bigr)
=
\sum_{r=1}^m p^{m-r}(U_r-V_r)
\in\mathbb Z^2.
\]
The map $W\mapsto\Pi_m(W)$ is injective on $\mathcal A_p^m$, so
$U\neq V$ implies $v\neq0$. Write
$v=(v_1,v_2),$ and $d:=\gcd(|v_1|,|v_2|),$
and define
\[
a:=\frac1d(v_2,-v_1)\in\mathbb Z^2.
\]
Then $a$ is primitive and perpendicular to $v$. Thus,
\begin{equation}\label{eq:blockcollision}
a\cdot\Pi_m(U)
=
a\cdot\Pi_m(V).
\end{equation}

Finally, for every $W\in\mathcal A_p^m$,
$p^m\Pi_m(W)
=
\sum_{r=1}^m p^{m-r}W_r
\in
\{0,\ldots,p^m-1\}^2.$
In particular
$|v|_\infty\leq p^m-1.$
Since $d\geq1$,
\begin{equation}\label{eq:aheight}
|a|_\infty
=
\frac{|v|_\infty}{d}
\leq
p^m-1.
\end{equation}

\medskip

\noindent
\textbf{Step 2: From one window to a positive density of windows.}

Let
$N=rm,
\, r\geq1,$
and partition $\{1,\ldots,N\}$ into the consecutive blocks
$$I_i:=\{(i-1)m+1,\ldots,im\},
\,
1\leq i\leq r.$$
For each $1\leq i \leq r$, define
\[
X_i(\omega)
:=
(\omega_{(i-1)m+1},\ldots,\omega_{im})
\in\mathcal A_p^m, \text{ and  }
Y_i(\omega)
:=
p^m a\cdot\Pi_m(X_i(\omega))
=
\sum_{k=1}^m p^{m-k}
\,a\cdot\omega_{(i-1)m+k}.
\]
Let
$\mathscr H_i
:=
\sigma(\omega_n:n\notin I_i)$
be the $\sigma$-algebra generated by all coordinates outside $I_i$. Recalling the event $G$ from Step 1,
set
$G_i:=\sigma^{-(i-1)m}G.$
By stationarity,
$G_i\in\mathscr H_i,$ and
$\widetilde\nu(G_i)=\widetilde\nu(G)=\kappa.$
The shifted form of \eqref{eq:twoWordsConditional} gives, on $G_i$,
\begin{equation}\label{eq:twoWordsShifted}
\widetilde\nu(X_i=U\mid\mathscr H_i)\geq\delta,
\text{ and }
\widetilde\nu(X_i=V\mid\mathscr H_i)\geq\delta.
\end{equation}
Define the $\mathscr G_i$-measurable functions
$\alpha_i
:=
\widetilde\nu(X_i=U\mid\mathscr H_i),$ and
$\beta_i
:=
\widetilde\nu(X_i=V\mid\mathscr H_i).$
Thus, for $\widetilde\nu$-almost every $\omega\in G_i$,
$\alpha_i(\omega)\geq\delta,$ and
$\beta_i(\omega)\geq\delta.$
Moreover, by \eqref{eq:blockcollision},
$p^m a\cdot\Pi_m(U)
=
p^m a\cdot\Pi_m(V).$
So the two events $X_i=U$ and $X_i=V$ give the same value of $Y_i$.

Fix $\omega\in G_i$ and condition on $\mathscr H_i$. The events
$X_i=U$ and $X_i=V$ have conditional probabilities
$\alpha_i(\omega)$ and $\beta_i(\omega)$, respectively, and are
identified by $Y_i$. Their contribution therefore gives
\[
\mathrm H_{\widetilde\nu}
(X_i\mid\mathscr H_i,Y_i)
\geq
\int_{G_i}
\left[
-\alpha_i\log\frac{\alpha_i}{\alpha_i+\beta_i}
-\beta_i\log\frac{\beta_i}{\alpha_i+\beta_i}
\right]
\,d\widetilde\nu.
\]
Since $\alpha_i,\beta_i\geq\delta$ on $G_i$, the integrand is at least
$2\delta\log2$. Hence
\[
\mathrm H_{\widetilde\nu}
(X_i\mid\mathscr H_i,Y_i)
\geq
2\delta\log2\,\widetilde\nu(G_i)
=
2\kappa\delta\log2.
\]
Setting
$c:=2\kappa\delta\log2>0,$
we obtain
\begin{equation}\label{eq:slotentropy}
\mathrm H_{\widetilde\nu}
(X_i\mid\mathscr H_i,Y_i)
\geq c,
\qquad
1\leq i\leq r.
\end{equation}
The constant $c$ depends only on the fixed scale-$m$ conditional
collision obtained in Step~1, and in particular is independent of
$i$, $r$, and $N$.

\medskip

\noindent
\textbf{Step 3: Linear entropy loss at scale $N$.}
Recall that $N=rm$. Define the random variable
\[
Q_{a,N}:\mathcal A_p^{\mathbb Z}\to\mathbb Z,
\text{ by } 
Q_{a,N}(\omega)
:=
p^N a\cdot\Pi_N(\omega)
=
\sum_{n=1}^N p^{N-n}a\cdot\omega_n.
\]
Thus $p^{-N}Q_{a,N}(\omega)$ is precisely the $N$-digit truncation of
$a\cdot\Pi(\omega)$.

Fix $1\leq i\leq r$. By the definition of $Y_i$, the contribution of
the coordinates in $I_i$ to $Q_{a,N}$ is
$p^{N-im}Y_i.$
Hence
\begin{equation}\label{eq:Qa-decomposition}
Q_{a,N}
=
C_i+p^{N-im}Y_i,
\end{equation}
where
$C_i
:=
\sum_{\substack{1\leq n\leq N\\ n\notin I_i}}
p^{N-n}a\cdot\omega_n.$
Since $C_i$ depends only on coordinates outside $I_i$, it is
$\mathscr H_i$-measurable. Therefore, conditional on
$\mathscr H_i$, the random variables $Q_{a,N}$ and $Y_i$ determine
each other. Therefore,
\begin{equation}\label{eq:QiYi}
\mathrm H_{\widetilde\nu}
\bigl(
X_i\mid\mathscr H_i,Q_{a,N}
\bigr)
=
\mathrm H_{\widetilde\nu}
\bigl(
X_i\mid\mathscr H_i,Y_i
\bigr)
\geq c,
\end{equation}
where the last inequality is \eqref{eq:slotentropy}.

Now consider
$\mathscr G_N
:=
\sigma(\omega_n:n\leq0\text{ or }n>N)$
 Conditioning on $\mathscr G_N$ and applying the
chain rule gives
\begin{align*}
\mathrm H_{\widetilde\nu}
(\omega_1^N\mid Q_{a,N})
&\geq
\mathrm H_{\widetilde\nu}
(\omega_1^N\mid Q_{a,N},\mathscr G_N)
\\
&=
\sum_{i=1}^r
\mathrm H_{\widetilde\nu}
\bigl(
X_i
\mid
X_1,\ldots,X_{i-1},Q_{a,N},\mathscr G_N
\bigr).
\end{align*}
For each $i$, the $\sigma$-algebra $\mathscr H_i$ contains
$\mathscr G_N$ and all blocks $X_j$ with $j\neq i$. Thus conditioning
further on $\mathscr H_i$ can only decrease entropy, and
\eqref{eq:QiYi} yields
\[
\begin{split}
&\mathrm H_{\widetilde\nu}
\bigl(
X_i
\mid
X_1,\ldots,X_{i-1},Q_{a,N},\mathscr G_N
\bigr)
\\
&\qquad\geq
\mathrm H_{\widetilde\nu}
\bigl(
X_i\mid\mathscr H_i,Q_{a,N}
\bigr)
\geq c.
\end{split}
\]
Summing over $i=1,\ldots,r$ and using $r=N/m$, we obtain
\begin{equation}\label{eq:linearConditionalEntropy}
\mathrm H_{\widetilde\nu}
(\omega_1^N\mid Q_{a,N})
\geq
rc
=
\frac{c}{m}N,
\qquad
N\in m\mathbb N.
\end{equation}

Finally, $Q_{a,N}$ is a function of $\omega_1^N$, so the chain rule
gives
$\mathrm H_{\widetilde\nu}(Q_{a,N})
=
\mathrm H_{\widetilde\nu}(\omega_1^N)
-
\mathrm H_{\widetilde\nu}(\omega_1^N\mid Q_{a,N}).$
Since the entropy rate of the process is $h$,
$\mathrm H_{\widetilde\nu}(\omega_1^N)
=
Nh+o(N).$ 
Combining this with \eqref{eq:linearConditionalEntropy}, we conclude
that, along $N\in m\mathbb N$,
\begin{equation}\label{eq:QaEntropyDrop}
\mathrm H_{\widetilde\nu}(Q_{a,N})
\leq
N\left(h-\frac{c}{m}\right)+o(N).
\end{equation}

\medskip
\noindent
\textbf{Step 4: Entropy and dimension of the rational projection.}
We now use \eqref{eq:QaEntropyDrop} to bound the dimension of
$P_a\mu$. For $\omega\in\mathcal A_p^{\mathbb Z}$,
\[
P_a\Pi(\omega)
=
a\cdot\Pi(\omega)
=
p^{-N}Q_{a,N}(\omega)
+
\sum_{n>N}p^{-n}a\cdot\omega_n.
\]
Since $\omega_n\in\{0,\ldots,p-1\}^2$,
$|a\cdot\omega_n|
\leq
(p-1)|a|_1,$
and hence
\begin{equation}\label{eq:projectionTail}
\left|
P_a\Pi(\omega)-p^{-N}Q_{a,N}(\omega)
\right|
\leq
|a|_1p^{-N}.
\end{equation}

So, after fixing the value of $Q_{a,N}$, the random variable
$P_a\Pi(\omega)$ is confined to an interval of length at most
$2|a|_1p^{-N}.$
Such an interval meets at most
$M_a:=2|a|_1+2$
atoms of the $p^N$-adic partition $\mathcal D_N$. It follows that
\begin{equation}\label{eq:projectionEntropyCompare}
\mathrm H(P_a\mu,\mathcal D_N)
\leq
\mathrm H_{\widetilde\nu}(Q_{a,N})
+\log M_a.
\end{equation}

Combining \eqref{eq:QaEntropyDrop} and
\eqref{eq:projectionEntropyCompare}, and taking $N=rm$, gives
\[
\limsup_{r\to\infty}
\frac{
\mathrm H(P_a\mu,\mathcal D_{rm})
}{
rm\log p
}
\leq
\frac{h-c/m}{\log p}.
\]

Finally, since $a\neq0$, the map $P_a$ is a non-zero scalar multiple
of an orthogonal projection. Hence $P_a\mu$ is exact dimensional by
Theorem~\ref{thm:exact-dimensionality}. In particular,
\[
\dim(P_a\mu)
=
\lim_{N\to\infty}
\frac{
\mathrm H(P_a\mu,\mathcal D_N)
}{
N\log p
}.
\]
Since $h=s\log p$, the preceding subsequential estimate yields
\[
\dim(P_a\mu)
\leq
\frac{h-c/m}{\log p}
=
s-\frac{c}{m\log p}
<s.
\]
Together with \eqref{eq:aheight}, this proves
Proposition~\ref{prop:local-drop}.
\end{proof}

\subsection{Elementary properties of the bilateral uncertainty}
\label{subsec:convexDm}

We next record some elementary properties of the quantities $D_m$. In
particular, we show that their successive increments are nondecreasing.

\begin{lemma}\label{lem:Dm-properties}
The sequence 
$(D_m)_{m\geq0}$ is nondecreasing and convex.
Moreover, if
$D_m>0$ for some $m\geq1$ then 
\begin{equation}\label{eq:Dmlinear}
D_m
\geq
(m-m_*+1)D_{m_*},
\qquad m\geq m_*:=\min\{m\geq1:D_m>0\}.
\end{equation}
\end{lemma}

\begin{proof}
For $m\geq1$, the chain rule gives
\begin{align*}
D_m
&=
\Ent_{\widetilde\nu}
\bigl(\omega_1\mid\omega_{\leq0},\omega_{>m}\bigr)
+
\Ent_{\widetilde\nu}
\bigl(\omega_2^m\mid\omega_{\leq1},\omega_{>m}\bigr)\\
&=
\Ent_{\widetilde\nu}
\bigl(\omega_1\mid\omega_{\leq0},\omega_{>m}\bigr)
+D_{m-1},
\end{align*}
where the second equality follows from stationarity. Hence
\begin{equation}\label{eq:dmincrement}
D_m-D_{m-1}
=
\Ent_{\widetilde\nu}
\bigl(\omega_1\mid\omega_{\leq0},\omega_{>m}\bigr)
\geq0.
\end{equation}
Thus $(D_m)$ is nondecreasing.

Moreover, since
$\sigma(\omega_{\leq0},\omega_{>m+1})
\subseteq
\sigma(\omega_{\leq0},\omega_{>m}),$
conditioning decreases entropy, and therefore
$D_{m+1}-D_m
\geq
D_m-D_{m-1}.$
Thus the successive increments of $(D_m)$ are nondecreasing, so $(D_m)$
is convex.

The second assertion is a standard consequence of the first.
\end{proof}

\subsection{Finite predictive information}
\label{subsec:finiteE}

For $j\geq0$, define
\begin{equation}\label{eq:Ejdef}
E^{(j)}
:=
\I_{\widetilde\nu}
\bigl(
\omega_{\leq0};\omega_{>j}
\bigr),
\end{equation}
and write
$E:=E^{(0)}
=
\I_{\widetilde\nu}
\bigl(
\omega_{\leq0};\omega_{>0}
\bigr).$
Here mutual information between infinite coordinate collections is
understood in the extended sense, as the supremum of the mutual
informations of their finite subblocks. Thus $E$ and $E^{(j)}$ may
a priori be infinite.

We call $E$ the \emph{predictive information} of the process, and
$E^{(j)}$ the predictive information across a gap of length $j$.

For every $j\geq0$,
\begin{equation}\label{eq:Egapcompare}
E^{(j)}
\leq
E
\leq
E^{(j)}
+
\mathrm H_{\widetilde\nu}(\omega_1^j)
\leq
E^{(j)}+2j\log p.
\end{equation}
Indeed, $\omega_{>j}$ is a subcollection of $\omega_{>0}$, so
monotonicity of mutual information gives
$E^{(j)}\leq E.$
On the other hand, since
$\omega_{>0}=(\omega_1^j,\omega_{>j}),$
the chain rule gives
\[
E
=
E^{(j)}
+
\I_{\widetilde\nu}
\bigl(
\omega_{\leq0};
\omega_1^j
\mid
\omega_{>j}
\bigr)
\leq
E^{(j)}
+
\mathrm H_{\widetilde\nu}(\omega_1^j).
\]
Finally,
$\mathrm H_{\widetilde\nu}(\omega_1^j)
\leq
\log|\mathcal A_p^j|
=
2j\log p.$
In particular,
$E^{(j)}<\infty
\Longleftrightarrow
E<\infty$
for every $j\geq0$.

The following proposition relates predictive information to the
bilateral uncertainties $D_m$.

\begin{proposition}\label{prop:finite-predictive}
Suppose that $h>0$, or, equivalently, that $\dim\mu>0$.

\begin{enumerate}[label=\textup{(\roman*)}]
\item
If $E<\infty$, then $D_m>0$ for all sufficiently large $m$.
Therefore, $\mu$ admits a dimension-dropping rational projection.

\item
If $E^{(j)}<\infty$ for some $j\geq0$, then the primitive vector
$a\in\mathbb Z^2\setminus\{0\}$ giving the dimension drop may be
chosen so that
\begin{equation}\label{eq:heightEj}
|a|_\infty
\leq
p^{\,j+\lfloor E^{(j)}/h\rfloor+1}-1.
\end{equation}
\end{enumerate}
\end{proposition}

Recall that a $\sigma$-invariant probability measure
$\nu\in\mathcal P(\mathcal A_p^{\mathbb N})$ is
\emph{quasi-Bernoulli} if there exists $C\geq1$ such that
\begin{equation*}
C^{-1}\nu([U])\nu([V])
\leq
\nu([UV])
\leq
C\nu([U])\nu([V])
\end{equation*}
whenever $U,V$ are finite words with $\nu([UV])>0$. We now deduce
Theorem~\ref{cor:quasi-bernoulli-drop} from
Proposition~\ref{prop:finite-predictive}.

\begin{proof}[Proof of Theorem~\ref{cor:quasi-bernoulli-drop}]
Let $C\geq1$ be the  constant from
\eqref{eq:quasiBernoulli}. For $n,m\geq1$, set
\[
E_{n,m}
:=
\I_{\widetilde\nu}
\bigl(
\omega_{-n+1}^0;\omega_1^m
\bigr).
\]
For $U\in\mathcal A_p^n$ and $V\in\mathcal A_p^m$, stationarity gives
\[
\widetilde\nu
\bigl(
\omega_{-n+1}^0=U,\,
\omega_1^m=V
\bigr)
=
\nu([UV]),
\]
while the corresponding marginal probabilities are
$\nu([U])$ and $\nu([V])$. Therefore
\[
E_{n,m}
=
\sum_{\substack{
U\in\mathcal A_p^n,\;
V\in\mathcal A_p^m\\
\nu([UV])>0
}}
\nu([UV])
\log
\frac{\nu([UV])}
     {\nu([U])\nu([V])}.
\]
By the upper bound in \eqref{eq:quasiBernoulli},
$\frac{\nu([UV])}
     {\nu([U])\nu([V])}
\leq C$
for every term appearing in the sum. Hence
$E_{n,m}
\leq
\sum_{U,V}\nu([UV])\log C
=
\log C.$

As $n,m\to\infty$, the blocks $\omega_{-n+1}^0$ and $\omega_1^m$
exhaust the whole past $\omega_{\leq0}$ and future $\omega_{>0}$,
respectively. By the definition of mutual information for infinite
coordinate collections,
\[
E
=
\sup_{n,m\geq1}E_{n,m}
\leq
\log C
<
\infty.
\]
The conclusion now follows from
Proposition~\ref{prop:finite-predictive}\textup{(i)}.
\end{proof}

\begin{proof}[Proof of Proposition~\ref{prop:finite-predictive}]
We first prove \textup{(i)}. By stationarity,
$\I_{\widetilde\nu}
\bigl(
\omega_{\leq m};\omega_{>m}
\bigr)
=
E.$
Since
$\omega_{\leq m}
=
(\omega_{\leq0},\omega_1^m),$
the chain rule for mutual information gives
\begin{align*}
E
&=
\I_{\widetilde\nu}
\bigl(
\omega_{\leq0};\omega_{>m}
\bigr)
+
\I_{\widetilde\nu}
\bigl(
\omega_1^m;\omega_{>m}
\mid\omega_{\leq0}
\bigr)
\\
&\geq
\I_{\widetilde\nu}
\bigl(
\omega_1^m;\omega_{>m}
\mid\omega_{\leq0}
\bigr).
\end{align*}
On the other hand,
\begin{align*}
\I_{\widetilde\nu}
\bigl(
\omega_1^m;\omega_{>m}
\mid\omega_{\leq0}
\bigr)
&=
\mathrm H_{\widetilde\nu}
\bigl(
\omega_1^m\mid\omega_{\leq0}
\bigr)
-
\mathrm H_{\widetilde\nu}
\bigl(
\omega_1^m\mid\omega_{\leq0},\omega_{>m}
\bigr)
\\
&=
mh-D_m,
\end{align*}
where
$\mathrm H_{\widetilde\nu}
\bigl(
\omega_1^m\mid\omega_{\leq0}
\bigr)
=
mh$
follows from stationarity and the definition of the entropy rate.
Therefore
\begin{equation}\label{eq:DmFromE}
D_m\geq mh-E.
\end{equation}
If $E<\infty$, then $D_m>0$ whenever $m>E/h$. Proposition~\ref{prop:local-drop}
therefore gives a primitive
$a\in\mathbb Z^2\setminus\{0\}$
such that
$\dim(P_a\mu)<\dim\mu.$
This proves \textup{(i)}.

For \textup{(ii)}, suppose that $E^{(j)}<\infty$ for some $j\geq0$, and set
$k
:=
\left\lfloor\frac{E^{(j)}}{h}\right\rfloor+1.$
Then
$kh>E^{(j)}.$
We claim that
$D_{j+k}>0.$
Suppose, to the contrary, that $D_{j+k}=0$. Then
\[
\mathrm H_{\widetilde\nu}
\bigl(
\omega_1^{j+k}
\mid
\omega_{\leq0},\omega_{>j+k}
\bigr)
=
0,
\]
and hence, since $\omega_1^k$ is a function of $\omega_1^{j+k}$,
\[
\mathrm H_{\widetilde\nu}
\bigl(
\omega_1^k
\mid
\omega_{\leq0},\omega_{>j+k}
\bigr)
=
0.
\]
Since
$\mathrm H_{\widetilde\nu}
\bigl(
\omega_1^k\mid\omega_{\leq0}
\bigr)
=
kh,$
we obtain
\begin{align*}
kh
&=
\I_{\widetilde\nu}
\bigl(
\omega_1^k;\omega_{>j+k}
\mid\omega_{\leq0}
\bigr)
\\
&\leq
\I_{\widetilde\nu}
\bigl(
\omega_{\leq k};\omega_{>j+k}
\bigr)
\\
&=
E^{(j)},
\end{align*}
where the last equality follows from stationarity. This contradicts
$kh>E^{(j)}$, proving that
$D_{j+k}>0.$

Applying Proposition~\ref{prop:local-drop} at scale $j+k$, we obtain a
primitive
$a\in\mathbb Z^2\setminus\{0\}$
such that
$\dim(P_a\mu)<\dim\mu$
and
$|a|_\infty
\leq
p^{j+k}-1
=
p^{\,j+\lfloor E^{(j)}/h\rfloor+1}-1.$
This proves \eqref{eq:heightEj} and completes the proof.
\end{proof}

\bibliographystyle{plain}
\bibliography{bib}

\begin{thebibliography}{10}

\bibitem{AlgomHost}
Amir Algom.
\newblock A simultaneous version of {Host}'s equidistribution theorem.
\newblock {\em Trans. Amer. Math. Soc.}, 373(12):8439--8462, 2020.

\bibitem{Algom2021}
Amir Algom.
\newblock Actions of diagonal endomorphisms on conformally invariant measures
  on the $2$-torus.
\newblock {\em Monatsh. Math.}, 195(4):545--564, 2021.

\bibitem{AlgomRodriguezHertzWang2026}
Amir Algom, Federico Rodriguez~Hertz, and Zhiren Wang.
\newblock Smooth projections of self-similar measures.
\newblock {\em arXiv preprint arXiv:2607.15635}, 2026.

\bibitem{algom2024plane}
Amir Algom, Federico Rodriguez~Hertz, and Zhiren Wang.
\newblock Spectral gaps and {F}ourier decay for self-conformal measures on the
  plane.
\newblock {\em Transactions of the American Mathematical Society},
  379(6):3953--3991, 2026.

\bibitem{AlgomShmerkin2024}
Amir Algom and Pablo Shmerkin.
\newblock On the dimension of orthogonal projections of self-similar measures.
\newblock {\em J. Lond. Math. Soc.}, 112(1):e70245, 2025.

\bibitem{BaranyKaenmakiKolossvary2026}
Bal{\'a}zs B{\'a}r{\'a}ny, Antti K{\"a}enm{\"a}ki, and Istv{\'a}n
  Kolossv{\'a}ry.
\newblock Projections of self-affine sets onto lines.
\newblock {\em arXiv preprint arXiv:2607.14740}, 2026.

\bibitem{barany2023scaling}
Bal{\'a}zs B{\'a}r{\'a}ny, Antti K{\"a}enm{\"a}ki, Aleksi Py{\"o}r{\"a}l{\"a},
  and Meng Wu.
\newblock Scaling limits of self-conformal measures.
\newblock {\em arXiv preprint arXiv:2308.11399}, 2023.

\bibitem{BaranySimonSolomyak2023}
Bal{\'a}zs B{\'a}r{\'a}ny, K{\'a}roly Simon, and Boris Solomyak.
\newblock {\em Self-similar and self-affine sets and measures}, volume 276 of
  {\em Mathematical Surveys and Monographs}.
\newblock American Mathematical Society, Providence, RI, 2023.

\bibitem{bishop2013fractal}
Christopher~J. Bishop and Yuval Peres.
\newblock {\em Fractals in probability and analysis}, volume 162 of {\em
  Cambridge Studies in Advanced Mathematics}.
\newblock Cambridge University Press, Cambridge, 2017.

\bibitem{Bruce2019jin}
Catherine Bruce and Xiong Jin.
\newblock Projections of {G}ibbs measures on self-conformal sets.
\newblock {\em Nonlinearity}, 32(2):603--621, 2019.

\bibitem{bruce2022furstenberg}
Catherine Bruce and Xiong Jin.
\newblock Furstenberg sumset conjecture and {M}andelbrot percolations.
\newblock {\em arXiv preprint arXiv:2211.16410}, 2022.

\bibitem{Carbery2009}
Anthony Carbery.
\newblock Large sets with limited tube occupancy.
\newblock {\em J. Lond. Math. Soc. (2)}, 79(2):529--543, 2009.

\bibitem{CarberySoriaVargas2007}
Anthony Carbery, Fernando Soria, and Ana Vargas.
\newblock Localisation and weighted inequalities for spherical {Fourier} means.
\newblock {\em J. Anal. Math.}, 103:133--156, 2007.

\bibitem{Chen2016}
Changhao Chen.
\newblock Distribution of random {Cantor} sets on tubes.
\newblock {\em Ark. Mat.}, 54(1):39--54, 2016.

\bibitem{FalconerJin2014}
Kenneth~J. Falconer and Xiong Jin.
\newblock Exact dimensionality and projections of random self-similar measures
  and sets.
\newblock {\em J. Lond. Math. Soc. (2)}, 90(2):388--412, 2014.

\bibitem{Fan2002measures}
Ai-Hua Fan, Ka-Sing Lau, and Hui Rao.
\newblock Relationships between different dimensions of a measure.
\newblock {\em Monatsh. Math.}, 135(3):191--201, 2002.

\bibitem{Farkas2016}
{\'A}bel Farkas.
\newblock Projections of self-similar sets with no separation condition.
\newblock {\em Israel J. Math.}, 214(1):67--107, 2016.

\bibitem{feng2009dimension}
De-Jun Feng and Huyi Hu.
\newblock Dimension theory of iterated function systems.
\newblock {\em Comm. Pure Appl. Math.}, 62(11):1435--1500, 2009.

\bibitem{FengXie2025}
De-Jun Feng and Yu-Hao Xie.
\newblock Dimensions of orthogonal projections of typical self-affine sets and
  measures.
\newblock {\em arXiv preprint arXiv:2502.04000}, 2025.

\bibitem{furstenberg1967disjointness}
Harry Furstenberg.
\newblock Disjointness in ergodic theory, minimal sets, and a problem in
  diophantine approximation.
\newblock {\em Theory of Computing Systems}, 1(1):1--49, 1967.

\bibitem{Grillenberger1973}
Christian Grillenberger.
\newblock Constructions of strictly ergodic systems. {I}. given entropy.
\newblock {\em Z. Wahrscheinlichkeitstheorie verw. Gebiete}, 25(4):323--334,
  1973.

\bibitem{hochman2014self}
Michael Hochman.
\newblock On self-similar sets with overlaps and inverse theorems for entropy.
\newblock {\em Ann. of Math. (2)}, 180(2):773--822, 2014.

\bibitem{hochman2009local}
Michael Hochman and Pablo Shmerkin.
\newblock Local entropy averages and projections of fractal measures.
\newblock {\em Ann. of Math. (2)}, 175(3):1001--1059, 2012.

\bibitem{HochmanShmerkinNormal}
Michael Hochman and Pablo Shmerkin.
\newblock Equidistribution from fractal measures.
\newblock {\em Invent. Math.}, 202(1):427--479, 2015.

\bibitem{Host1995normal}
Bernard Host.
\newblock Nombres normaux, entropie, translations.
\newblock {\em Israel J. Math.}, 91(1-3):419--428, 1995.

\bibitem{Johnson1992}
Aimee S.~A. Johnson.
\newblock Measures on the circle invariant under multiplication by a
  nonlacunary subsemigroup of the integers.
\newblock {\em Israel Journal of Mathematics}, 77(1-2):211--240, 1992.

\bibitem{jordan2019dimension}
Thomas Jordan and Ariel Rapaport.
\newblock Dimension of ergodic measures projected onto self-similar sets with
  overlap.
\newblock {\em Proceedings of the London Mathematical Society},
  122(2):191--206, 2021.

\bibitem{Lindenstrauss2001}
Elon Lindenstrauss.
\newblock $p$-adic foliation and equidistribution.
\newblock {\em Israel J. Math.}, 122:29--42, 2001.

\bibitem{mattila1999geometry}
Pertti Mattila.
\newblock {\em Geometry of sets and measures in {E}uclidean spaces}, volume~44
  of {\em Cambridge Studies in Advanced Mathematics}.
\newblock Cambridge University Press, Cambridge, 1995.
\newblock Fractals and rectifiability.

\bibitem{MorrisSert2025}
Ian~D. Morris and {\c{C}}a{\u{g}}r{\i} Sert.
\newblock Projections of self-affine fractals.
\newblock {\em Inventiones Mathematicae}, 2026.
\newblock To appear.

\bibitem{NazarovPeresShmerkin2009}
Fedor Nazarov, Yuval Peres, and Pablo Shmerkin.
\newblock Convolutions of {Cantor} measures without resonance.
\newblock {\em Israel J. Math.}, 187:93--116, 2012.

\bibitem{OrnsteinWeiss1975}
Donald~S. Ornstein and Benjamin Weiss.
\newblock Every transformation is bilaterally deterministic.
\newblock {\em Israel Journal of Mathematics}, 21(2--3):154--158, 1975.

\bibitem{Orponen2015}
Tuomas Orponen.
\newblock On the tube occupancy of sets in $\mathbb{R}^d$.
\newblock {\em Int. Math. Res. Not. IMRN}, (19):9815--9831, 2015.

\bibitem{OrponenRutar2026}
Tuomas Orponen and Alex Rutar.
\newblock Visibility problem in the plane.
\newblock {\em arXiv preprint arXiv:2606.06965}, 2026.

\bibitem{PeresShmerkin2009}
Yuval Peres and Pablo Shmerkin.
\newblock Resonance between cantor sets.
\newblock {\em Ergodic Theory Dynam. Systems}, 29(1):201--221, 2009.

\bibitem{Aleksi2024reso}
Aleksi Py\"or\"al\"a.
\newblock The dimension of projections of planar diagonal self-affine measures.
\newblock {\em Ann. Fenn. Math.}, 50(1):59--78, 2025.

\bibitem{Aleksi2025pablo}
A.~Pyörälä, P.~Shmerkin, V.~Suomala, and M.~Wu.
\newblock Covering the {Sierpiński} carpet with tubes.
\newblock {\em Israel Journal of Mathematics}, 265:769--799, 2025.

\bibitem{Rapaport2017}
Ariel Rapaport.
\newblock A self-similar measure with dense rotations, singular projections and
  discrete slices.
\newblock {\em Adv. Math.}, 321:529--546, 2017.

\bibitem{Rudolph1990}
Daniel~J. Rudolph.
\newblock $\times 2$ and $\times 3$ invariant measures and entropy.
\newblock {\em Ergodic Theory Dynam. Systems}, 10(2):395--406, 1990.

\bibitem{ShmerkinSuomala2015}
Pablo Shmerkin and Ville Suomala.
\newblock Sets which are not tube null and intersection properties of random
  measures.
\newblock {\em J. Lond. Math. Soc. (2)}, 91(2):405--422, 2015.

\bibitem{Wu2025Countable}
Meng Wu.
\newblock Projection theorems with countably many exceptions and applications
  to the exact overlaps conjecture.
\newblock {\em arXiv preprint arXiv:2503.21923}, 2025.

\end{thebibliography}

\end{document}